\documentclass[12pt,a4paper,leqno]{amsart}

\usepackage[top=1.1in, bottom=1.3in, left=1.5in, right=1.5in]{geometry}

\usepackage[ansinew]{inputenc}
\usepackage[T1]{fontenc}
\usepackage[english]{babel}
\usepackage{amsthm}
\usepackage{amsfonts}         
\usepackage{amsmath}
\usepackage{amssymb}
\usepackage{breqn}
\usepackage{bm}
\usepackage{url}
\usepackage{esint}
\usepackage{fancyhdr}

\newcommand{\C}{\mathbb{C}}

\newcommand{\N}{\mathbb{N}}

\newcommand{\Z}{\mathbb{Z}}

\newcommand{\A}{\mathcal{A}}
\newcommand{\B}{\mathcal{B}}
\newcommand{\CC}{\mathcal{C}}

\newcommand\pprec{\prec\mkern-5mu\prec}

\theoremstyle{plain}
\newtheorem{theorem}{Theorem}
\newtheorem{lemma}[theorem]{Lemma}
\newtheorem{prop}[theorem]{Proposition}

\newtheorem{conj}{Conjecture}

\theoremstyle{remark}
\newtheorem{remark}{Remark}

\numberwithin{equation}{section}
\begin{document}

\title{On the largest prime factor of $n^2+1$}

\author{Jori Merikoski}

\address{Department of Mathematics and Statistics, University of Turku, FI-20014 University of Turku,
Finland}
\email{jori.e.merikoski@utu.fi}

\begin{abstract}
We show that the largest prime factor of $n^2+1$ is infinitely often greater than $n^{1.279}$. This improves the result of de la Bret\`eche and Drappeau (2020) who obtained this with $1.2182$ in place of $1.279.$ The main new ingredients in the proof are a new Type II estimate and using this estimate by applying Harman's sieve method. To prove the Type II estimate we use the bounds of Deshouillers and Iwaniec on linear forms of Kloosterman sums. We also show that conditionally on Selberg's eigenvalue conjecture the exponent $1.279$ may be increased to $1.312.$
\end{abstract}

\maketitle
\tableofcontents

\section{Introduction}
An outstanding open problem in number theory is to prove that there are infinitely many primes of the form $n^2+1$. To approximate this we may consider the largest prime factor of integers of the form $n^2+1,$ as was done by Chebyshev already in the 19th century (cf. the introduction in \cite{hooley} for the prehistory of this problem). In 1967 Hooley \cite{hooley} proved that the largest prime factor of $n^2+1$ is infinitely often at least $n^{1.10014\dots}$ by applying the Weil bound for Kloosterman sums.  Deshouillers and Iwaniec \cite{DI} showed in 1982 that the largest prime factor of $n^2+1$ is at least $n^{1.202468\dots}$ infinitely often. Their improvement came as an application of their bounds for linear forms of Kloosterman sums \cite{DI2}. In 2020 de la Bret\`eche and Drappeau \cite{BD} improved the exponent to 1.2182 by making use of the result of Kim and Sarnak \cite[Appendix 2]{KS} towards Selberg's eigenvalue conjecture.

We will show a new Type II estimate (Proposition \ref{typeiiprop} below) and use this by applying Harman's sieve method to improve the previous results:
\begin{theorem} \label{maint}
The largest prime factor of $n^2+1$ is greater than $n^{1.279}$ for infinitely many integers $n$.
\end{theorem}

\begin{remark} The proof of Theorem \ref{maint} uses the deep bound of Kim and Sarnak \cite[Appendix 2]{KS}. Using just the classical Selberg's $3/16$-Theorem our argument gives a result with the exponent  $1.279$ replaced by $1.23$.
\end{remark}

We also obtain a new conditional result (improving the exponent $\sqrt{3/2}-\epsilon \geq 1.2247$ of Deshouillers and Iwaniec \cite[Section 8]{DI}):
\begin{theorem} \label{selbergt}
Assuming Selberg's eigenvalue conjecture the exponent $1.279$ in Theorem \ref{maint} may be increased to $1.312$.
\end{theorem}

\begin{remark} As is usual with Harman's sieve, the exact limit of the method is hard to determine and would require extensive numerical computations. The exponents in both of the above theorems could still be slightly improved by optimizing the sieve more carefully but we do not pursue this issue here for the sake of simplifying presentation.
\end{remark}

\begin{remark} By using similar arguments as in \cite{BD}, \cite{DFI}, and \cite{hooley} it should be possible to generalise our result from $n^2+1$ to polynomials $n^2-d$ where $d$ is not a perfect square. 
\end{remark}

\subsection{Sketch of the proof}
Similarly as in \cite{BD}, \cite{DI}, and \cite{hooley}, we will use Chebyshev's device to detect large prime factors, that is, we use the elementary fact that
\begin{align*}
\sum_{m} \Lambda(m)  \sum_{\substack{\ell \sim x \\ \ell^2+1 \equiv 0 \, \, (m)}} 1 = \sum_{\substack{\ell \sim x }} \sum_{m | \ell^2+1} \Lambda(m) = \sum_{\substack{\ell \sim x }} \log(\ell^2+1) = 2x\log x + O(x)
\end{align*}
so that if $P_x$ denotes the largest prime factor of $\ell^2+1$ for $\ell \sim x$, then
\begin{align*}
 \sum_{p \leq P_x} \log p \sum_{\substack{\ell \sim x \\ \ell^2+1 \equiv 0 \, \, (p)}} 1 \geq (2+o(1)) x\log x. 
\end{align*}
Hence, to get a lower bound for $P_x$ we require upper bounds for sums of the type
\begin{align} \label{heur}
\sum_{p \sim P} \sum_{\substack{\ell \sim x \\ \ell^2+1 \equiv 0 \, \, (p)}} 1,
\end{align}
where $P \leq  x^\varpi $ with $\varpi$ corresponding to the exponent in Theorem \ref{maint}.

Deshouillers and Iwaniec \cite{DI} use a linear sieve upper bound for the sum (\ref{heur}), and the main point in their work is to obtain strong Type I information, that is, asymptotic formulas for sums of the form
\begin{align*}
\sum_{d \leq D} \lambda_d \sum_{\substack{m\sim P \\ m \equiv 0 \,\, (d)}}  \sum_{\substack{\ell \sim x \\ \ell^2+1 \equiv 0 \, \, (m)}} 1,
\end{align*}
where $\lambda_d$ are bounded coefficients. The level of distribution obtained in \cite[Section 7]{DI} is $D=x^{1-\epsilon} P^{-1/2}$, which improved the level $D=x^{1-\epsilon}P^{-3/4}$ in Hooley's work \cite{hooley} (the conditions $m\sim P$ and $\ell \sim x$ need to be replaced by smooth coefficients but let us ignore this detail for the moment). De la Bret\`eche and Drappeau \cite{BD} improve the level of distribution to $D=x^{1/(2-4\theta)-\epsilon}P^{-\theta/(1-2\theta)}$, where $\theta \geq 0$ is any admissible exponent in the Ramanujan-Selberg conjecture. Note that from Selberg's $3/16$-Theorem we know $\theta =1/4$ is admissible which gives the same the level of distribution as in the work of Deshouillers and Iwaniec \cite{DI}. The exponent 1.2182 in \cite{BD} follows from using the result of Kim and Sarnak \cite{KS} that $\theta=7/64$ is admissible.

We will use a combination of Harman's sieve method \cite{harman} and the linear sieve to give an improved upper bound for (\ref{heur}) for some ranges of $P$ (see the beginning of Section \ref{buchsec} for a heuristic explanation of Harman's sieve). Our sieve has similarities also to the sieve used by Duke, Friedlander, and Iwaniec in \cite{DFI}. For the sieve we need to obtain Type II information, that is, an asymptotic formula for sums of the form
\begin{align} \label{typeiiheur}
\sum_{\substack{m \sim M \\ n \sim N}} a_m b_n  \sum_{\substack{\ell \sim x \\ \ell^2+1 \equiv 0 \, \, (mn)}} 1,
\end{align}
where $MN=P$ and $a_m$ and $b_n$ are bounded coefficients. Type II sums of this form are also considered in the works Iwaniec \cite{iwaniec}, Lemke Oliver \cite{lemke} and more recently in \cite[Th\'eor\`eme 5.2]{BD}, but they are not applied to the problem of the largest prime factor of $n^2+1$.  Our Proposition \ref{typeiiprop} gives an improvement on \cite[Th\'eor\`eme 5.2]{BD}. The proof of our Type II estimate is given in Section \ref{typeiisection}. The sieve argument is carried out in Section \ref{sievesection}, using the Type I information proved in \cite{BD}. 

Our proof of the Type II information is inspired by  the arguments in \cite{DI} and \cite{DFI}. The key ingredient in the proof is an estimate for linear forms of Kloosterman sums of the form
\begin{align} \label{sumklooster}
\sum_{r} \sum_{\substack{m \sim \bm{M}\\ n \sim \bm{N}}}A_{m,r} B_{n,r} \sum_{(c,r)=1} g(m,n,c,r) S(m \overline{r},\pm n;c),
\end{align}
for some nice smooth function $g$.  Unfortunately both of the coefficients $A_{m,r}$ and $B_{n,r}$ depend on $r$, so that we are unable to make use of the average over the `level variable' $r$ (cf. \cite[Theorem 10]{DI2} for such a result). Similarly as the results in \cite{BD}, our Type II information will depend on the smallest eigenvalue $\lambda_1(r)=1/4-\theta_r^2$ for the Hecke congruence subgroups $\Gamma_0(r)$ (cf. \cite[Section 1]{DI2} for precise definitions). Selberg's eigenvalue conjecture famously states that $\lambda_1(\Gamma) \geq 1/4$ for any congruence subgroup $\Gamma$. The current best lower bound is the result of Kim and Sarnak \cite[Appendix 2]{KS} that $\lambda_1(\Gamma) \geq 1/4-(7/64)^2$, which we will apply with the estimate of Deshouillers and Iwaniec \cite[Theorem 9]{DI2} to obtain a bound for the sum (\ref{sumklooster}) individually for each $r$. 

For a more detailed sketch of the proof of the Type II estimate we refer to the begininning of Section \ref{typeiisection}. Unfortunately we can handle Type II sums only in the range $P < x^{153/128}$, so that for $x^{153/128} < P < x^\varpi$ we cannot improve on the upper bound of \cite{BD}. Note that even for $P=x^{1+\epsilon}$ a good upper bound for (\ref{heur}) is highly nontrivial, in fact, for $P=x^{1+\epsilon}$ the linear sieve upper bound is off by a factor of $4+O(\epsilon)$. 

In the last section we outline some open problems whose resolution would lead to further progress on the largest prime factor of $n^2+1$.

\subsection{Notations}
We use the following asymptotic notations: for functions $f$ and $g$ with $g$ positive, we write $f \ll g$ or $f= \mathcal{O}(g)$ if there is a constant $C$ such that $|f|  \leq C g.$ The notation $f \asymp g$ means $g \ll f \ll g.$ The constant may depend on some parameter, which is indicated in the subscript (e.g. $\ll_{\epsilon}$).
We write $f=o(g)$ if $f/g \to 0$ for large values of the variable. For variables we write $n \sim N$ meaning $N<n \leq 2N$. 

It is convenient for us to define
\begin{align*}
A \pprec B
\end{align*}
to mean $A \, \ll_\epsilon x^{\epsilon} B.$ A typical bound we use is $\tau_k(n) \pprec 1$  for $n \ll x$, where $\tau_k$ is the $k$-fold divisor function. We say that an arithmetic function $f$ is divisor bounded if $|f(n)| \ll \tau_k(n)$ for some $k$.

We let $\eta >0$ denote a sufficiently small constant, which may be different from place to place. For example, $A \ll x^{-\eta}B$ means that the bound holds for some $\eta >0.$

For a statement $E$ we denote by $1_E$ the characteristic function of that statement. For a set $A$ we use $1_A$ to denote the characteristic function of $A.$ 

We also define $P(w):= \prod_{p\leq w} p,$ where the product is over primes. 

We let $e(x):= e^{2 \pi i x}$ and $e_q(x):= e(x/q)$ for any integer $q \geq 1$. For integers $a,$ $b$, and $q \geq 1$ with $(b,q)=1$ we define $e_{q}(a/b) := e(a\overline{b}/q)$. For Kloosterman sums we use the standard notation
\begin{align*}
S(a,b;c):= \sum_{\substack{n \,\, (c)\\ (n,c)=1}} e_c(an+b/n).
\end{align*}
 
\subsection{Acknowledgements} I am grateful to my supervisor Kaisa Matom\"aki for useful discussions,  comments, and support. I also express my gratitude to Emmanuel Kowalski for helpful discussions. I also wish to thank Philippe Michel for bringing the article \cite{BD} to my attention. I am also grateful for the anonymous referee for comments. During the work the author was funded by UTUGS Graduate School.

\section{The sieve} \label{sievesection}
In this section we will state the arithmetical information (Propositions \ref{typeiprop} and \ref{typeiiprop} below) and apply them with Harman's sieve method \cite{harman} and the linear sieve to give a proof of Theorem \ref{maint}. We also sketch the proof of Theorem \ref{selbergt} by indicating how the proof of Theorem \ref{maint} needs to be modified.

\subsection{Set up} Our notations will be mostly similar to those of \cite{DI}. For $x \geq 1$, let $b$ denote a non-negative $C^{\infty}$-smooth function, supported on $[x,2x]$, whose derivatives satisfy for all $j \geq 0$
\begin{align*}
b^{(j)}(x) \ll_j x^{-j}.
\end{align*}
For any integer $d \geq 1,$ define
\begin{align*}
|\A_d| \,\, := \sum_{n^2+1 \equiv 0 \,\, (d)} b(n) \quad \quad \text{and} \quad\quad X := \int b(\xi) \, d \xi.
\end{align*}

If $P_x$ denotes the greatest prime factor of $\prod_{x \leq n \leq 2x}(n^2+1)$, then by using the Chebysev-Hooley method similarly as in \cite[Section 2]{DI} we find
\begin{align} \label{cheby}
S(x):= \sum_{x < p \leq P_x } |\A_p| \log p = X \log x + O(x).
\end{align}
Therefore, we require an upper bound of $S(x)$ to get a lower bound for $P_x$. We first split the sum using a smooth dyadic partition of unity similarly as in \cite[Section 3]{DI}
\begin{align*}
S(x) = \sum_{\substack{x \leq P \leq P_x \\ P= 2^j x}} S(x,P) + O(x),
\end{align*} 
where
\begin{align*}
S(x,P) = \sum_{P \leq p \leq 4P} \psi_P(p) |\A_p| \log p
\end{align*}
for some $C^{\infty}$-smooth functions $\psi_P$ supported on $[P,4P]$ satisfying $\psi_P^{(\ell)}(\xi) \ll_\ell P^{-\ell}$ for all $\ell\geq 0.$ 

Compared to \cite{BD} and \cite{DI}, we will improve on their upper bound for $S(x,P)$ but only for  $x \leq P < x^{153/128}$. This is because only in this range we are able to prove a new bilinear estimate (Proposition \ref{typeiiprop}). To see how to use this new arithmetic information, we first note that in \cite{BD} and \cite{DI} the upper bound for $S(x)$ is obtained  by using the linear sieve. Since the linear sieve is neutral with respect to applications of Buchstab's identity, we may apply Buchstab's indentity as we please to obtain Type II sums which we now have an asymptotic formula instead of just upper and lower bounds of the linear sieve, thus improving on the linear sieve bound. A similar principle also appears in the sieve of Duke, Friedlander, and Iwaniec in \cite{DFI}. By applying Harman's sieve method the use of the linear sieve can be completely avoided in some ranges (cf. \cite[Sections 3.5 and 3.8]{harman} for further discussion on the relation between Harman's sieve and the linear sieve).

For $P \geq x^{153/128}$ we are unable to obtain new information and we just apply the same argument as in \cite[Section 8]{DI} to get an upper bound for $S(x,P)$. In the end we sum over the dyadic ranges $x \leq P \leq x^\varpi$ to determine the largest $\varpi$ for which we can show that
\begin{align*}
\sum_{x < p \leq x^\varpi} |\A_p| \log p \leq (1-\epsilon) X \log x.
\end{align*}

As usual with Harman's sieve method, we have to calculate numerical upper bounds for multi-dimensional integrals. These integrals are computed using Python 3.7, and the links to the codes can be found at the end of this section.

\subsection{Arithmetic information}
Let us define
\begin{align*}
\rho(m):= | \{ \nu \in \Z/m\Z: \,\, \nu^2+1 \equiv 0 \,\, (m) \}  |.
\end{align*}
\begin{remark} In \cite{DI} this is denoted by $\omega(m)$ but we reserve the symbol $\omega$ for the Buchstab function.
\end{remark}

We say that an arithmetic function $f(n)$ is divisor bounded if $|f(n)|  \ll \tau_k(n)$ for some $k \geq 1.$ From the work of de la Bret\`eche and Drappeau we know the following linear estimate \cite[Section 8.4]{BD} (it is stated there for bounded coefficients $\lambda_d$ but the same holds also for coefficients which are divisor bounded since the saving in the error term is a power of $x$). 
\begin{prop} \emph{(Type I information, de la Bret\`eche-Drappeau).} \label{typeiprop} Let $\theta=7/64$. Let $x \leq P=x^\alpha \leq x^{2-\eta}$ and
\begin{align*}
D:=x^{1/(2-4\theta)-\eta}P^{-\theta/(1-2\theta)}= x^{(1-2\theta\alpha)/(2-4\theta)-\eta}=x^{(32-7\alpha)/50-\eta}.
\end{align*}
Suppose that $D \ll x^{2-\eta}/P.$ Let $\lambda_d$ be any divisor bounded coefficients. Then
\begin{align*}
\sum_{d \leq D} \lambda_d \sum_{m \equiv 0 \,\, (d)} |\A_m| \psi_P(m) \log m = X \sum_{d \leq D} \lambda_d \sum_{m \equiv 0 \,\, (d)} \frac{\rho(m)}{m} \psi_P(m) \log m + O (x^{1-\eta}).
\end{align*}
\end{prop}
\begin{remark} We use $\eta$ to denote a positive constant which can be taken to be arbitrarily small and which may be different from place to place (similarly as in \cite{harman}). Hence, the above proposition says that for every small $\eta_1>0$ there exists $\eta_2>0$ such that if $P\leq x^{2-\eta_1}$ and $D=x^{(32-7\alpha)/50-\eta_1} \ll x^{2-\eta_1}/P$, then the claimed asymptotic formula holds with an error term $O(x^{1-\eta_2})$. The exact dependence between the various $\eta$'s which appear is irrelevant in our arguments.
\end{remark}

In Section \ref{typeiisection} we will show the following bilinear estimate which improves on \cite[Th\'eor\`eme 5.2]{BD}:
\begin{prop} \emph{(Type II information).} \label{typeiiprop} Let $\theta= 7/64.$ Let $P=x^\alpha$ for some $\alpha \geq 1$, and let $MN=P$ for $M,N \geq 1$. Let $a_m$ and $b_n$ be divisor bounded coefficients such that $b_n$ is supported on square-free integers.  Then
\begin{align*}
\sum_{\substack{m \sim M \\ n \sim N}} a_m b_n |\A_{mn}| \psi_P(mn) \log mn = X \sum_{\substack{m \sim M \\ n \sim N}}   \frac{a_m b_n \rho(mn)}{mn} \psi_P(mn) \log mn  + O (x^{1-\eta}).
\end{align*}
if one of the following holds:
\\
\textbf{\emph{(i)}} 
\begin{align*}
x^{\alpha-1+\eta} \ll N \ll  x^{(2-2\theta-\alpha)/3-\eta}=x^{(57-32\alpha)/96 - \eta}  .
\end{align*}
\\
\textbf{\emph{(ii)}}\emph{(Duke-Friedlander-Iwaniec+de la Bret\`eche-Drappeau)}  $b_n$ is supported on primes and
\begin{align*}
x^{2(\alpha-1)+\eta} \ll N \ll  x^{(4-(3+2\theta)\alpha)/(3-6\theta)}=x^{(128-103\alpha)/75-\eta}. 
\end{align*}
\end{prop}
\begin{remark} The part (i) gives a non-trivial range for $N$ if $\alpha < 5/4-\theta/2 =153/128=1.195\dots$
\end{remark}
\begin{remark} The exponent $\theta=7/64$ corresponds to the smallest eigenvalues $\lambda_1(q)$ on the Hecke congruence subgroups $\Gamma_0(q), \, q\geq 1,$ by $\lambda_1(q)= 1/4-\theta_q^2$ (cf. \cite[Section 1]{DI2} for precise definitions). Under Selberg's eigenvalue conjecture we could set $\theta=0.$ That $\theta_q \leq 7/64$ follows from a deep result of Kim and Sarnak \cite[Appendix 2]{KS}.
\end{remark}

\begin{remark} The part (ii) is almost a direct consequence of combining the argument in \cite[Section 5]{DFI} with \cite[Lemme 8.3, part 1]{BD}. The upper limit is better than (i) only in the range $\alpha < 2671/2496 = 1.070\dots$. Notice that for $\theta=0$ our part (i) gives a better result in the full range.
\end{remark}

\begin{remark} By similar arguments as in \cite{iwaniec} and \cite{lemke}, in \cite[Th\'eor\`eme 5.2]{BD} de la Bret\`eche and Drappeau use the dispersion method to handle Type II sums for
\begin{align*}
x^{\alpha-1+\eta} \ll N \ll x^{\alpha(1-2\theta)/(7-6\theta)-\eta}
\end{align*}
but this is weaker than Proposition \ref{typeiiprop}(i).
\end{remark}

\subsection{Fundamental Proposition}
For integers $d \geq 1$ denote
\begin{align} \label{asieve}
S(\A(P)_d,z) := \sum_{\substack{(n,P(z))=1}} |\A_{dn}| \psi_P(dn) \log (dn),
\end{align}
so that (denoting $\A(P)=\A(P)_1$ when $d=1$)
\begin{align*}
S(x,P) = S(\A(P), 2 \sqrt{P}).
\end{align*}
Let us also define the expected value of $S(\A(P)_d,z)$
\begin{align} \label{bsieve}
S(\B(P)_d,z) := X \sum_{(n,P(z))=1} \frac{\rho(dn)}{dn} \psi_P(dn) \log(dn).
\end{align}
For $d=1$ denote $\B(P)=\B(P)_1$.

For the next Proposition we note that $(2-2\theta-\alpha)/3 > 2(\alpha-1)$ exactly if $\alpha < 249/224 =1.11\dots$ We can combine Propositions \ref{typeiprop} and \ref{typeiiprop} by using a variant of the argument in \cite[Chapter 3]{harman} to get
\begin{prop} \emph{(Fundamental Proposition I).} \label{funprop} Let $P=x^\alpha$ for $1\leq \alpha < 249/224-2\eta.$ Let $D$ be as in Proposition \ref{typeiprop} and set
\begin{align*}
 U:=Dx^{1-\alpha-\eta}=x^{(1-2\theta\alpha)/(2-4\theta)-\alpha+1-2\eta},
\end{align*}
and
\begin{align*}
\sigma := \max \bigg\{\frac{2-2\theta-\alpha}{3}-\eta, \frac{4-(3+2\theta)\alpha}{3-6\theta}-\eta\bigg \}.
\end{align*}
Let $\lambda_u$ be divisor-bounded coefficients. Then
\begin{align*}
\sum_{u \leq U } \lambda_u S(\A(P)_u,x^\sigma) = \sum_{u \leq U} \lambda_u S(\B(P)_u,x^\sigma) + O(x^{1-\eta}).
\end{align*}
\end{prop}
\begin{proof}
Using the M\"obius function to detect $(n,P(x^\sigma))=1$, we have 
\begin{align*}
\sum_{u \leq U } \lambda_u S(\A(P)_u,x^\sigma)  &= \sum_{u \leq U } \sum_{d | P(x^\sigma)} \sum_n \lambda_u \mu(d) |\A_{udn}| \psi_P(udn)\log(udn) \\
&= \Sigma_I(\A(P)) + \Sigma_{II}(\A(P)), 
\end{align*}
where
\begin{align*}
 \Sigma_I(\A) &:= \sum_{u \leq U } \sum_{\substack{d | P(x^\sigma)\\ d \leq x^{\alpha-1+\eta}} } \sum_n \lambda_u \mu(d) |\A_{udn}| \psi_P(udn)\log(udn) \quad \quad \text{and} \\
 \Sigma_{II}(\A) & := \sum_{u \leq U } \sum_{\substack{d | P(x^\sigma)\\ d > x^{\alpha-1+\eta}} } \sum_n \lambda_u \mu(d) |\A_{udn}| \psi_P(udn)\log(udn).
\end{align*}
 Similarly, we can write
\begin{align*}
\sum_{u \leq U } \lambda_u S(\B(P)_u,x^\sigma) =\Sigma_I(\B(P)) + \Sigma_{II}(\B(P)).
\end{align*}

For the first pair of sums, since $du \leq x^{\alpha-1+\eta}U = D,$ we have by Proposition \ref{typeiprop}
\begin{align*}
\Sigma_I(\A(P)) = \Sigma_I(\B(P))  + O(x^{1-\eta}).
\end{align*}

In the second pair of sums we have (writing $d=q_1q_2\cdots q_k$)
\begin{align*}
\Sigma_{II}(\A(P)) = \sum_{k \ll \log x}(-1)^k \sum_{u \leq U} \sum_{\substack{q_k < \cdots < q_1 \leq x^{\sigma} \\ q_1 \cdots q_k > x^{\alpha-1+\eta} }} \lambda_u |\A_{uq_1 \cdots q_k n}| \psi_P(uq_1 \cdots q_k n)\log(u q_1 \cdots q_kn).
\end{align*}
For every $q_1\cdots q_k$ there exists a unique $\ell \leq k$ such that
\begin{align*}
q_1 \cdots q_\ell \geq x^{\alpha-1+\eta} \quad \quad \text{and} \quad \quad q_1 \cdots q_{\ell-1} < x^{\alpha-1+\eta}.
\end{align*}
Hence, writing $n':=q_1 \cdots q_\ell$ and $m:=un q_{\ell+1} \cdots q_k$, and using Perron's formula to remove the cross-condition $q_\ell < q_{\ell+1}$ (cf. \cite[Chapter 3.2]{harman}),  we can partition  $\Sigma_{II}(\A(P))$ into
\begin{align*}
 \sum_{k \ll \log x}(-1)^k \sum_{\ell \leq k} \sum_m \sum_{\substack{n'=q_1 \cdots q_\ell \geq x^{\alpha-1+\eta} \\ q_1 \cdots q_{\ell-1} < x^{\alpha-1+\eta} \\ q_\ell < \cdots < q_1 \leq x^\sigma}} a_m b_{n'} |\A_{mn'}| \psi_P(mn') \log mn'
\end{align*}
with $b_{n'}$ supported on square-free integers. A similar partition applies to $\Sigma_{II}(\B(P))$. 

If $\ell =1$, then $x^{\alpha-1+\eta} \leq q_1 \leq x^\sigma$, so that we have an asymptotic formula by combining Proposition \ref{typeiiprop}(i) and (ii) if $\alpha < 2671/2496$, and for $\alpha \geq 2671/2496$ simply using part (i).

If $\ell > 1$, then we have $q_1 \cdots q_\ell \leq x^{\alpha-1+\eta} q_\ell \leq x^{(2-2\theta-\alpha)/3-\eta}$ (since $q_\ell < q_1< x^{\alpha-1+\eta}$ and $2(\alpha-1) < (2-2\theta-\alpha)/3-3\eta$ for $\alpha < 249/224-2\eta$), so that we may apply  Proposition \ref{typeiiprop}(i) to get an asymptotic formula. Summing over $\ell$ and $k$ we obtain
\begin{align*}
\Sigma_{II}(\A(P)) = \Sigma_{II}(\B(P)) + O(x^{1-\eta}).
\end{align*} 
\end{proof}

We note that $(2-2\theta-\alpha)/3>\alpha-1$ precisely if $\alpha < 153/128.$ By a similar argument we obtain the following variant of the previous proposition 
\begin{prop} \emph{(Fundamental Proposition II).} \label{funprop2} Let $P=x^\alpha$ for $1\leq \alpha < 153/128-2\eta.$ Let $D$ be as in Proposition \ref{typeiprop} and set
\begin{align*}
 U:=Dx^{1-\alpha-\eta}=x^{(1-2\theta\alpha)/(2-4\theta)-\alpha+1-2\eta},
\end{align*}
and
\begin{align*}
\gamma := \frac{2-2\theta-\alpha}{3} - \alpha+1 - 2\eta.
\end{align*}
Let $\lambda_u$ be divisor-bounded coefficients. Then
\begin{align*}
\sum_{u \leq U } \lambda_u S(\A(P)_u,x^\gamma) = \sum_{u \leq U} \lambda_u S(\B(P)_u,x^\gamma) + O(x^{1-\eta}).
\end{align*}
\end{prop}
\begin{proof}
The only difference to the proof of Proposition \ref{funprop} is that this time in $\Sigma_{II}(\A(P))$ combining
\begin{align*}
q_1 \cdots q_\ell \geq x^{\alpha-1+\eta} \quad \quad \text{and} \quad \quad q_1 \cdots q_{\ell-1} < x^{\alpha-1+\eta}
\end{align*}
with $q_\ell < x^\gamma$ we get $q_1\cdots q_\ell < x^{\alpha-1+\eta+\gamma} <  x^{(2-2\theta-\alpha)/3  -\eta}$, so that we may use Proposition \ref{typeiiprop}(i) to get an asymptotic formula.
\end{proof}

We also need a lemma for transforming sums over almost-primes into integrals which can be evaluated numerically. Let $\omega(u)$ denote the Buchstab function (cf. \cite[Chapter 1]{harman} for the properties below, for instance), so that by the Prime Number Theorem for $y^{\epsilon} < z < y$ 
\begin{align} \label{buchasymp}
\sum_{y < n \leq 2y} 1_{(n,P(z))=1} = (1+o(1)) \omega \left(\frac{\log y}{\log z} \right) \frac{y}{\log z}.
\end{align}
 Note that for $1< u \leq 2$ we have $\omega(u)=1/u.$ In the numerical computations we will use the following bounds for the Buchstab function (cf. \cite[Lemma 20]{jia})
\begin{align} \label{buchbound}
\omega(u) \, \begin{cases} = 0, &u < 1 \\
= 1/u, & 1 \leq u < 2 \\
= (1+\log(u-1))/u, &2 \leq u < 3 \\
\leq 0.5644, &  3 \leq u < 4 \\
\leq 0.5617, & u \geq 4 \\
\geq 0.5607, & u \geq 2.47.
\end{cases}
\end{align}

In the lemma below we assume that the range $\mathcal{U}\subset [x^{\eta},Px^{-\eta}]^{k}$ is sufficiently well-behaved, e.g. an intersection of sets of the type $\{ \boldsymbol{u}: u_i < u_j \}$ or $\{\boldsymbol{u}: V <  f(u_1, \dots,u_k) < W\}$ for some polynomial $f$ and some fixed $V,W.$

\begin{lemma} \label{bilemma} Let $\mathcal{U} \subset [x^{\eta},Px^{-\eta}]^{k} $ and $P=x^\alpha$. Then
\begin{align*}
\sum_{(q_1, \dots , q_k) \in \mathcal{U}} S(\B(P)_{q_1, \dots, q_k},q_k) = (1+o(1))X \int \psi_P(u)  \frac{du}{u} \alpha \int \omega (\alpha,\boldsymbol{\beta }) \frac{d\beta_1 \cdots d\beta_k}{\beta_1\cdots\beta_{k-1}\beta_k^2},
\end{align*}
where the integral is over the range $\{\boldsymbol{\beta}: \, (x^{\beta_1}, \dots, x^{\beta_k}) \in \mathcal{U}\}$, and 
\begin{align*}
 \omega(\alpha,\boldsymbol{\beta}):= \omega\bigg(\frac{\alpha-\beta_1-\cdots 
 -\beta_k}{\beta_k}\bigg).
\end{align*} 
\end{lemma}
\begin{proof}
By definition the left-hand side in the lemma is equal to
\begin{align*}
\sum_{(q_1, \dots , q_k) \in \mathcal{U}} X \sum_m 1_{(m,P(q_k))=1} \frac{\rho(q_1\cdots q_k m)}{q_1 \cdots q_k m} \psi_P(q_1 \cdots q_k m) \log (q_1 \cdots q_k m).
\end{align*}
Note that the function $\rho(m)$ is multiplicative and $\rho(p) = 2\cdot 1_{p \equiv 1 \, (4)}$ for primes $p > 2.$ Hence, for $(m,P(x^\eta))=1$ we can replace $\rho(m)$ by 1 with negligible error by equidistribution of primes in arithmetic progressions. Therefore, by (\ref{buchasymp}) and by the Prime Number Theorem we have
\begin{align*}
& \sum_{(q_1, \dots , q_k) \in \mathcal{U}} S(\B(P)_{q_1, \dots, q_k},q_k) \\
&= \sum_{(q_1, \dots , q_k) \in \mathcal{U}} X \sum_m 1_{(m,P(q_k))=1} \frac{1}{q_1 \cdots q_k m} \psi_P(q_1 \cdots q_k m) \log (q_1 \cdots q_k m) \\
&= (1+o(1))X \int \psi_P(u) \log u \frac{du}{u} \sum_{(q_1, \dots , q_k) \in \mathcal{U}} \frac{1}{q_1\cdots q_k \log q_k} \omega \left( \frac{\log(P/(q_1\cdots q_k))}{\log q_k} \right) \\
&= (1+o(1))X \int \psi_P(u) \log u \frac{du}{u} \\
& \hspace{60pt}  \sum_{(n_1,\dots,n_k ) \in \mathcal{U}} \frac{1}{n_1\cdots n_k (\log n_1) \dots (\log n_{k-1} )\log^2 n_k} \omega \left( \frac{\log(P /(n_1\cdots n_k))}{\log n_k} \right) \\
&= (1+o(1))X \int \psi_P(u) \log u \frac{du}{u} \\
& \hspace{60pt} \int_{\mathcal{U}}  \omega \left( \frac{\log(P/(u_1\cdots u_k))}{\log u_k} \right)   \frac{du_1\cdots du_k}{u_1\cdots u_k (\log u_1) \dots (\log u_{k-1} )\log^2 u_k}\\
&=  (1+o(1))X \int \psi_P(u)  \frac{du}{u} \alpha \int \omega (\alpha,\boldsymbol{\beta }) \frac{d\beta_1 \cdots d\beta_k}{\beta_1\cdots\beta_{k-1}\beta_k^2}
\end{align*}
by the change of variables $u_j=x^{\beta_j}$ and by inserting $\log u = (1+o(1))\alpha \log x$.
\end{proof}
\begin{remark} We refer to the factor $\alpha \int \omega(\alpha,\boldsymbol{\beta}) \frac{ d \beta_1\cdots d\beta_k }{\beta_1\cdots\beta_{k-1}\beta_k^2}$ as the deficiency of the corresponding sum. 
\end{remark} 

For the linear sieve (cf. \cite[Chapter 11]{OdC}) we let $F(s),f(s)$ denote the continuous solution to the system of delay-differential equations
\begin{align*}
\begin{cases} (sF(s))' = f(s-1) \\
(sf(s))' = F(s-1)
\end{cases}
\end{align*}
with the initial condition
\begin{align*}
\begin{cases} sF(s) = 2e^{\gamma}, & \text{if} \, \, 1 \leq s \leq 3 \\
sf(s) = 0, & \text{if} \, \, s\leq 2.
\end{cases}.
\end{align*}
Here $\gamma$ is the Euler-Mascheroni constant.  We require the following
\begin{lemma}\textbf{\emph{(Linear sieve upper bound).}} \label{linearlemma} Let $D$ be as in Proposition \ref{typeiprop}. For $P=x^\alpha$ and for any $x^\eta < z < D$ we have
\begin{align*}
S(\A(P),z) \leq (1+o(1)) X \int \psi_P(u)\frac{du}{u} \frac{\alpha \log x}{e^\gamma \log z} F\bigg( \frac{\log D}{\log z}\bigg).
\end{align*}
\end{lemma}
\begin{proof}
Let $\lambda_d$ denote the sieve weights of the upper bound linear sieve  \cite[Chapter 11]{OdC}) with level of distribution $D$. Then
\begin{align*}
S(\A(P),z)  \leq \sum_{\substack{d\leq  D \\ d| P(z)}}\lambda_d \sum_{m \equiv 0 \,\, (d)}|\A_m| \psi_P(m) \log m = X \sum_{d \leq D} \lambda_d \sum_{\substack{d\leq  D \\ d| P(z)}} \frac{\rho(m)}{m} \psi_P(m) \log m + O (x^{1-\eta})
\end{align*}
by Proposition \ref{typeiprop}. The sum on the right-hand side can now be evaluated by using \cite[Theorem 11.12]{OdC} and the same argument as in \cite[Section 8]{DI}, which leads to the result.
\end{proof}

\subsection{Buchstab decompositions} \label{buchsec}
The general idea of Harman's sieve is to use Buchstab's identity to decompose the sum $S(\CC(P),2\sqrt{P})$  (in parallel for $\CC(P)=\A(P)$ and $\CC(P)=\B(P)$) into a sum of the form $\sum_k \epsilon_k S_k(\CC(P)),$ where $\epsilon_k \in \{-1,1\},$ and  $S_k(\CC(P)) \geq 0$ are sums over almost-primes.  Since we are interested in an upper bound, for $\CC(P)=\A(P)$ we can insert the trivial estimate $S_k(\A(P)) \geq 0$ for any $k$ such that the sign $\epsilon_k =-1;$ these sums are said to be discarded. For the remaining $k$ we will obtain an asymptotic formula by using Propositions \ref{typeiiprop} and \ref{funprop} (in some cases with $\epsilon_k=1$ we will use the linear sieve upper bound (Lemma \ref{linearlemma}) but let us ignore this for now). That is, if $\mathcal{K}$ is the set of indices that are discarded, then
\begin{align*}
S(\A(P), 2\sqrt{P})&= \sum_k \epsilon_k S_k(\A(P)) \leq \sum_{k \notin \mathcal{K}} \epsilon_k S_k(\A(P))  \\
&=(1+o(1)) \sum_{k \notin \mathcal{K}} \epsilon_k  S_k(\B(P)) =  (1+o(1))S(\B(P),2\sqrt{P}) +  \sum_{k \in \mathcal{K}}  S_k(\B(P)).
\end{align*} 
By the Prime Number Theorem we have
\begin{align*}
S(\B(P), 2 \sqrt{P}) = (1+o(1)) X \int \psi_P(u)\frac{du}{u}.
\end{align*}
The remaining sum $\sum_{k \in \mathcal{K}}  S_k(\B(P))$ we can estimate using Lemma \ref{bilemma}. Thus, we will obtain an upper bound of the form
\begin{align} \label{Gclaim}
S(\A(P),2\sqrt{P}) \leq (1+ G(\alpha))X \int \psi_P(u)\frac{du}{u}
\end{align}
for some non-negative function $G$ measuring the deficiency at range $P=x^\alpha.$ 

To relax the notations we will ignore factors of $x^{\eta}$ in the ranges of variables in this section, since their contribution to $G(\alpha)$ will be $O(\eta)$ which can be made arbitrarily small. 

We separate into five cases, $1\leq \alpha \leq 758/733,$ $758/733 \leq \alpha < 249/224$, $249/224 \leq \alpha <182/157  $,  $182/157 \leq \alpha < 153/128,$ and $\alpha > 153/128$.

\begin{remark} The range $\alpha < 249/224$ is where we can apply Proposition \ref{funprop}. For $\alpha < 182/157$ we will use Propostion \ref{funprop2}. For $182/157 \leq \alpha < 153/128$ we will use a combination of Proposition \ref{typeiiprop}(i) and the linear sieve upper bound. For $\alpha > 153/128$ we do not have any new information so that we just use the linear sieve similarly as in \cite{BD} and \cite{DI} to get an upper bound.
\end{remark}

\subsubsection{Case $1\leq \alpha < 758/733$} \label{alphasmallsection}
Let
\begin{align*}
\sigma := \frac{4-(3+2\theta)\alpha}{3-6\theta}-\eta
\end{align*}
(for $\alpha< 758/733$ part (ii) of Proposition \ref{typeiiprop} is stronger than (i)). Define $\xi$ by setting (recall Proposition \ref{funprop})
\begin{align*}
U=Dx^{1-\alpha-\eta}=x^{(1-2\theta\alpha)/(2-4\theta)-\alpha+1-2\eta}=:x^{\xi},
\end{align*}
 Let $\CC \in \{\A,\B\}$. By Buchstab's identity we have
\begin{align*}
S(\CC(P), 2 \sqrt{P})= S(\CC(P), x^\sigma) -  \sum_{x^\sigma < q \leq 2\sqrt{P}} S(\CC(P)_q, q).
\end{align*}
By Proposition \ref{funprop} we have an asymptotic formula for the first term. In the second sum we note that the implicit variable in $S(\CC(P)_q, q)$ (cf. $n$ in (\ref{asieve}) and (\ref{bsieve})) is of size $ x^{\alpha}/q$, so that for $q \gg x^{\alpha-2\sigma}$ the implicit variable runs over primes of size  $<x^{2\sigma}.$ Hence
\begin{align*}
\sum_{x^{\alpha-2\sigma} \ll q \leq U} S(\CC(P)_q, q) = \sum_{x^{\alpha-2\sigma} \ll q \leq U} S(\CC(P)_q, x^\sigma),
\end{align*}
so that we have an asymptotic formula by Proposition \ref{funprop} in this range. We note that this range is non-trivial precisely if
\begin{align*}
\alpha < 758/733 = 1.034\dots.
\end{align*}
The remaining part we just discard, which by Lemma  \ref{bilemma} gives us a deficiency
\begin{align} \label{g1alpha}
\alpha \int_{\sigma}^{\alpha-2\sigma} \omega(\alpha/\beta-1) \frac{d \beta}{\beta^2} + \alpha \int_{\xi}^{\alpha/2} \omega(\alpha/\beta-1) \frac{d \beta}{\beta^2}.
\end{align}

\subsubsection{Case $758/733 \leq \alpha < 249/224$}
Let
\begin{align*}
\sigma := \max \bigg\{\frac{2-2\theta-\alpha}{3}-\eta, \frac{4-(3+2\theta)\alpha}{3-6\theta}-\eta\bigg \}.
\end{align*}
By Buchstab's identity we have
\begin{align*}
S(\CC(P), 2 \sqrt{P})= S(\CC(P), x^\sigma) -  \sum_{x^\sigma < q \leq 2\sqrt{P}} S(\CC(P)_q, q).
\end{align*}
By Proposition \ref{funprop} we have an asymptotic formula for the first term.  The second sum we just discard, which by Lemma  \ref{bilemma} gives us a deficiency
\begin{align*}
\alpha \int_{\sigma}^{\alpha/2} \omega(\alpha/\beta-1) \frac{d \beta}{\beta^2}.
\end{align*}
Summing over the dyadic ranges $x < P=2^j x < x^{249/224}$ we obtain
\begin{align*}
\sum_{\substack{x \leq P \leq x^{249/224} \\ P= 2^j x}} S(x,P) \leq (25/224 + G_1+G_2+o(1)) X \log x ,
\end{align*}
where by (\ref{g1alpha})
\begin{align*}
G_1 := \int_1^{758/733} \alpha\bigg( \int_{\sigma}^{\alpha-2\sigma} \omega(\alpha/\beta-1) \frac{d \beta}{\beta^2} + \int_{\xi}^{\alpha/2} \omega(\alpha/\beta-1) \frac{d \beta}{\beta^2} \bigg) d\alpha < 0.01745
\end{align*}
and
\begin{align*}
G_2 := \int_{758/733}^{249/224} \alpha \int_{\sigma}^{\alpha/2} \omega(\alpha/\beta-1) \frac{d \beta}{\beta^2}  d\alpha < 0.11478.
\end{align*}

\subsubsection{Case $249/224 \leq \alpha < 182/157$}
From here on we let $\sigma := (2-2\theta-\alpha)/3$ (for $\alpha \geq 249/224$ part (i) of Proposition \ref{typeiiprop} is stronger than (ii)). Recall that in Proposition \ref{funprop2}
\begin{align*}
\gamma := \frac{2-2\theta-\alpha}{3} - \alpha+1 - 2\eta.
\end{align*}
By applying Buchstab's identity we get
\begin{align*}
S(\A(P), 2 \sqrt{P})&= S(\A(P), x^{\gamma}) -  \sum_{x^{\gamma} < q \leq 2\sqrt{P}} S(\A(P)_q, q).
\end{align*}
For the first term we have an asymptotic formula by Proposition \ref{funprop2}. In the second sum we get an asymptotic formula by Proposition \ref{typeiiprop}(i) in the part $x^{\alpha-1}< q <x^\sigma$. We discard the part with $x^\sigma < q < x^{\alpha/2}$, which gives us a deficiency
\begin{align*}
\alpha \int_{\sigma}^{\alpha/2} \omega(\alpha/\beta-1) \frac{d \beta}{\beta^2}.
\end{align*}
For the remaining part $x^\gamma < q \leq x^{\alpha-1}$ we apply Buchstab's identity twice to get
\begin{align*}
-\sum_{x^{\gamma} < q\leq x^{\alpha-1}} S(\A(P)_q, q) = -\sum_{x^{\gamma} < q\leq x^{\alpha-1}}& S(\A(P)_q, x^\gamma)  + \sum_{x^{\gamma} < q_2<q_1\leq x^{\alpha-1}} S(\A(P)_{q_1q_2}, x^\gamma) \\
&-  \sum_{x^{\gamma} < q_3<q_2<q_1\leq x^{\alpha-1}} S(\A(P)_{q_1q_2q_3}, q_3).
\end{align*}
Since $\alpha < 182/157$, we have $x^{2(\alpha-1)} < U$ so that for the first two sums we have an asymptotic formula by Proposition \ref{funprop2}. In the last sum we use Proposition \ref{typeiiprop}(i) to get an asymptotic formula whenever any combination of $q_1,q_2,q_3$ is in the Type II range $[x^{\alpha-1},x^\sigma]$ and we discard the rest. Thus,
 \begin{align*}
\sum_{\substack{x^{249/224} \leq P \leq x^{182/157} \\ P= 2^j x}} S(x,P) \leq \bigg(\frac{182}{157}-\frac{249}{224} + G_3+G_4+o(1)\bigg) X \log x ,
\end{align*}
where
\begin{align*}
G_3 := \int_{249/224}^{182/157} \alpha \int_{\sigma}^{\alpha/2} \omega(\alpha/\beta-1) \frac{d \beta}{\beta^2}  d\alpha < 0.093754.
\end{align*}
and
\begin{align*}
G_4:= \int f_4(\alpha, \bm{\beta})\alpha\omega \bigg( \frac{\alpha-\beta_1-\beta_2-\beta_3}{\beta_3}\bigg) \frac{d \beta_1 d\beta_2 d\beta_3}{\beta_1 \beta_2 \beta_3^2} d\alpha < 0.0057
\end{align*}
with $f_4$ the characteristic function of the four dimensional set
\begin{align*}
\bigg\{\frac{249}{224}< \alpha & <\frac{182}{157}, \, \gamma < \beta_3< \beta_2< \beta_1 < \alpha-1 \\
& \beta_1+\beta_2, \, \beta_1+\beta_3, \, \beta_2+\beta_3, \, \beta_1+\beta_2+\beta_3 \notin [\alpha-1, \sigma] \bigg \}
\end{align*}

\subsubsection{Case $182/157 \leq \alpha < 153/128$}
By applying Buchstab's identity we get
\begin{align*}
S(\A(P), 2 \sqrt{P})&= S(\A(P), x^{\alpha-1}) -  \sum_{x^{\alpha-1} < q \leq 2\sqrt{P}} S(\A(P)_q, q) \\
&\leq  S(\A(P), x^{\alpha-1}) -  \sum_{x^{\alpha-1} <q \leq  x^\sigma} S(\A(P)_q, q).
\end{align*}
For the first term we use the linear sieve upper bound (Lemma \ref{linearlemma}), while for the second term we have an asymptotic formula by Proposition \ref{typeiiprop}. Hence, by Lemmata \ref{bilemma} and \ref{linearlemma} we get an upper bound
\begin{align*}
S(\A(P), 2 \sqrt{P}) \leq  (G_5(\alpha)-G_6(\alpha)+o(1)) X \int \psi_P(u)\frac{du}{u},
\end{align*}
so that
\begin{align*}
\sum_{\substack{x^{182/157} \leq P \leq x^{153/128} \\ P= 2^j x}} S(x,P) \leq ( G_5-G_6+o(1)) X \log x ,
\end{align*}
where
\begin{align*}
G_5 := e^{-\gamma}\int_{182/157}^{153/128}\frac{\alpha}{\alpha-1} F\bigg( \frac{1-2\theta\alpha}{(2-4\theta)(\alpha-1)}\bigg) d\alpha =  4(1-2\theta)\int_{182/157}^{153/128}\frac{\alpha}{1-2\theta\alpha}  d\alpha  < 0.17877
\end{align*}
and
\begin{align*}
G_6 := \int_{182/157}^{153/128} \alpha \int_{\alpha-1}^{\sigma} \omega(\alpha/\beta-1) \frac{d\beta}{\beta^2} d\alpha > 0.016329.
\end{align*}

\begin{remark} Here also we could apply Buchstab's identity multiple times to generate more Type II sums, similarly as we did for $\alpha < 182/157$. However, for $\alpha > 182/157$ the width of our Type II information is $\gamma <0.048$ so that the gain from this would be fairly small (certainly less than $G_6$) so we ignore this to simplify the argument.
\end{remark}

\subsubsection{Case $\alpha > 153/128$}
In the range $P \geq x^{153/128}$ we do not have any new information, so that just using the linear sieve upper bound (Lemma \ref{linearlemma}) we obtain
\begin{align} \label{linearremainder}
\sum_{\substack{x^{153/128} \leq P \leq x^\varpi \\ P= 2^j x}} S(x,P) \leq \bigg(4(1-2\theta) \int_{153/128}^\varpi \frac{\alpha}{1-2\theta\alpha} d \alpha+o(1)\bigg) X \log x  .
\end{align}

\subsection{Conclusion of the proof of Theorem \ref{maint}}
Summing over the estimates we get
\begin{align*}
\sum_{\substack{x\leq P \leq x^{153/128}  \\ P= 2^j x}} S(x,P) \leq  (25/157+G+o(1)) X \log x,
\end{align*}
where
\begin{align*}
25/157 + G = 25/157+G_1+G_2+G_3+G_4+G_5-G_6 < 0.553361
\end{align*}
Combining this with (\ref{linearremainder}), we have
\begin{align*}
 \frac{1}{X \log x}\sum_{\substack{x\leq P \leq x^{1.279}  \\ P= 2^j x}} S(x,P) < 0.553361 +4(1-2\theta) \int_{153/128}^{1.279} \frac{\alpha}{1-2\theta\alpha} d \alpha = 0.997\dots < 1,
\end{align*}
which proves Theorem \ref{maint} since otherwise we reach a contradiction with the asymptotic (\ref{cheby}). \qed

\begin{remark} In comparison, just using the linear sieve upper bound gives
\begin{align*}
\sum_{\substack{x \leq P \leq x^{153/128} \\ P= 2^j x}} S(x,P) \leq \bigg(4(1-2\theta) \int_{1}^{153/128} \frac{\alpha}{1-2\theta\alpha} d \alpha+o(1)\bigg) X \log x < 0.8213 \cdot X \log x.
\end{align*}
\end{remark}

\begin{remark} The method in \cite{BD} and \cite{DI} gives an asymptotic formula for $S(x,P)$ for $P \leq x$, but for $P=x^{1+\epsilon}$ the upper bound is off by a factor of $4+O(\epsilon)$. In contrast, we get the correct upper bound for $P=x^{1+\epsilon}$. As $P=x^\alpha$ varies from $x$ to $x^{153/128}$ our method can be enhanced to give an upper bound which continuously increases from an asymptotic formula to the linear sieve upper bound (this would require a more careful handling of the part $182/157 \leq \alpha < 153/128$). This is in accordance with the general principle of Harman's sieve method that our sieve bounds should depend continuously on the quality of the arithmetic information.
\end{remark}

The Python 3.7 codes for computations of the Buchstab integrals are available at: 
\begin{tabular}{ c c c }
&\hspace{50pt} $G_1$ & \quad \quad \url{http://codepad.org/e2RiL3TM} \\
&\hspace{50pt} $G_2$ &\quad \quad \url{http://codepad.org/i2BOT07g}  \\
&\hspace{50pt} $G_3$ & \quad \quad\url{http://codepad.org/vMlImNKm}  \\
&\hspace{50pt} $G_4$ &\quad \quad \url{http://codepad.org/DOxewic3}  \\
&\hspace{50pt} $G_6$ &\quad \quad \url{http://codepad.org/IKZNttfN}  \\
\end{tabular}

\subsection{Proof of Theorem \ref{selbergt}}
The sieve follows the same recipe as the proof of Theorem \ref{maint}. Assuming Selberg's conjecture we may set $\theta=0$, so that $D=x^{1/2}$, $U=x^{3/2-\alpha}=x^\xi$, and $\sigma = (2-\alpha)/3$. The reader will verify that now the ranges corresponding to  the five ranges in the proof of Theorem 1 are $1\leq \alpha < 17/16$, $17/16 \leq \alpha < 8/7$, $8/7 \leq \alpha < 7/6$, $7/6 < \alpha < 5/4$ and $\alpha \geq 5/4.$ By a similar application of Buchstab's identities we get
\begin{align*}
\sum_{\substack{x\leq P \leq x^{5/4}  \\ P= 2^j x}} S(x,P) \leq  (1/6+F+o(1)) X \log x,
\end{align*}
where 
\begin{align*}
1/6 + F = 1/6+ F_1+F_2+F_3+F_4+F_5-F_6 < 0.679914 
\end{align*}
with 
\begin{align*}
F_1 &:= \int_1^{17/16} \alpha\bigg( \int_{\sigma}^{\alpha-2 \sigma} \omega(\alpha/\beta-1) \frac{d \beta}{\beta^2} + \int_{\xi}^{\alpha/2} \omega(\alpha/\beta-1) \frac{d \beta}{\beta^2} \bigg) d\alpha < 0.0287 \\
F_2 &:=  \int_{17/16}^{8/7} \alpha \int_{\sigma}^{\alpha/2} \omega(\alpha/\beta-1) \frac{d \beta}{\beta^2}  d\alpha < 0.08622 \\
F_3 &:=  \int_{8/7}^{7/6} \alpha \int_{\sigma}^{\alpha/2} \omega(\alpha/\beta-1) \frac{d \beta}{\beta^2}  d\alpha < 0.03107 \\
F_4 &:= \int f_4(\alpha, \bm{\beta})\alpha\omega \bigg( \frac{\alpha-\beta_1-\beta_2-\beta_3}{\beta_3}\bigg) \frac{d \beta_1 d\beta_2 d\beta_3}{\beta_1 \beta_2 \beta_3^2} d\alpha < 0.00011 \\
F_5 &:=  4 \int_{7/6}^{5/4}\alpha d\alpha  =29/72 \\
F_6 &:= \int_{7/6}^{5/4} \alpha \int_{\alpha-1}^{\sigma} \omega(\alpha/\beta-1) \frac{d\beta}{\beta^2} d\alpha > 0.035631
\end{align*}
with $f_4$ the characteristic function of the four dimensional set
\begin{align*}
\bigg\{8/7< \alpha & <7/6, \, \gamma < \beta_3< \beta_2< \beta_1 < \alpha-1 \\
& \beta_1+\beta_2, \, \beta_1+\beta_3, \, \beta_2+\beta_3, \, \beta_1+\beta_2+\beta_3 \notin [\alpha-1, \sigma] \bigg \}.
\end{align*}
We also have by the linear sieve (Lemma \ref{linearlemma})
\begin{align*}
\sum_{\substack{ x^{5/4} \leq P \leq x^{\varpi} \\ P= 2^j x}} S(x,P) \leq  \bigg(4\int_{5/4}^\varpi \alpha d \alpha +o(1)\bigg) X \log x.
\end{align*}
Combining the two estimates we have
\begin{align*}
 \frac{1}{X \log x}\sum_{\substack{x\leq P \leq x^{1.312}  \\ P= 2^j x}} S(x,P) < 0.679914  + 4\int_{5/4}^{1.312} \alpha d \alpha  = 0.997\dots < 1,
\end{align*}
which implies Theorem \ref{selbergt}. \qed

\section{Type II information} \label{typeiisection}
In this section we give a proof of Proposition \ref{typeiiprop}. Let us first give a non-rigorous sketch of the argument.

\subsection{Sketch of the argument}
Similarly as in \cite{iwaniec} and \cite{lemke}, in \cite[Th\'eor\`eme 5.2]{BD} de la Bret\`eche and Drappeau obtain asymptotic formulas for Type II sums by using the dispersion method of Linnik (cf. \cite[Section 8.3.3]{BD}). 

Our argument is more direct. We begin by applying the Poisson summation formula to evaluate $|\A_{mn}|$. For simplicity, let us assume that $(m,n)=1$ in the Type II sum in Proposition \ref{typeiiprop}. Then by the Poisson summation formula (Lemma \ref{poisson}) we can reduce the claim to showing that for $H=x^\epsilon P/x$ and for any bounded coefficients $c_h$ we have
\begin{align*}
\frac{1}{H} \sum_{1 \leq |h| \leq H} c_h \sum_{\substack{m \sim M \\ n \sim N \\ (m,n)=1}} a_m b_n \sum_{\substack{\nu \,\, (mn) \\ \nu^2+1 \equiv 0 \,\, (mn)}} e_{mn}(-h\nu) \, \ll x^{1-\eta}.
\end{align*}
\begin{remark} Note that the length of the exponential sum is $MN=P,$ while we need a bound that is a bit less than $x$. Thus, we need to save a power of $x$, the more the bigger $P$ is. Since we need to apply the Cauchy-Schwarz inequality in the proof, all savings are essentially halved. For this reason we are unable to get an estimate for large $P$.
\end{remark}
\begin{remark} For a fixed $h$ this sum is the same bilinear sum as in the work of Duke, Friedlander and Iwaniec \cite[Proposition 2]{DFI}. Note that in their work only a small saving over the trivial bound is required, that is a bound $\ll P^{1-\eta}$. In this case their method gives unconditionally the same range as one gets assuming Selberg's conjecture (ie. $x^{\eta} \ll N \ll x^{1/3-\eta}$). Our argument has a similar flavour to their proof, but in contrast we also make use of the average over the frequencies $h$.
 \end{remark}
When we apply Cauchy-Schwarz we would like to simplify matters by keeping the sum over $\nu^2+1\equiv 0 \, (mn)$ `outside' while keeping the sum over $n$ `inside'. To facilitate this, recall that $b_n$ is supported on square-free integers. Hence, if we denote
\begin{align*}
Q:=Q(m):= \prod_{\substack{2 \leq p\leq 2N \\ p \equiv 1,2 \, \, (4) \\ p\, \nmid\, m}} p,
 \end{align*}
then by the Chinese Remainder theorem we have (for $(m,n)=1$)
\begin{align*}
  \sum_{\substack{\nu^2+1 \equiv 0 \,\, (mn)}} e_{mn}(-h\nu) =\frac{\rho(n)}{\rho(Q)} \sum_{\substack{\nu^2+1 \equiv 0 \,\, (mQ)}} e_{mn}(-h\nu).
\end{align*}

Let $\psi_M(m)$ denote a $C^\infty$-smooth majorant of $1_{m \sim M}.$ By the Cauchy-Schwarz inequality and by expanding the square afterwards we obtain
\begin{align*}
\sum_{m \sim M} &a_m \frac{1}{\rho(Q)} \sum_{\substack{\nu^2+1 \equiv 0 \,\, (m Q)}} \frac{1}{H} \sum_{1 \leq |h| \leq H} c_h \sum_{\substack{ n \sim N \\ (m,n)=1}}  b_n \rho(n) e_{mn}(-h\nu)  \\
& \pprec M^{1/2} \bigg( \sum_{m} \psi_M(m)\frac{1}{H^2} \sum_{1\leq |h_1|,|h_2| \leq H} c_{h_1} \overline{c_{h_2}}    \sum_{\substack{n_1,n_2 \sim N \\ (m,n_1n_2)=1}} b_{n_1}   \overline{b_{n_2}} \\
& \hspace{100pt}  \frac{\rho(n_1)\rho(n_2)}{\rho(Q)} \sum_{\substack{\nu^2+1 \equiv 0 \,\, (mQ)}} e_{mn_1}(-h_1\nu) e_{mn_2}(h_2\nu)\bigg)^{1/2} \\ 
&\pprec M^{1/2} \bigg( \frac{1}{H^2} \sum_{1\leq |h_1|,|h_2| \leq H} c_{h_1} \overline{c_{h_2}}   \sum_{n_0\ll N} \rho(n_0) \sum_{\substack{n_1,n_2 \sim N/n_0 \\ (n_1,n_2)=1}} b_{n_0n_1}  \overline{b_{n_0 n_2}} \\
& \hspace{70pt} \sum_{(m,n_0n_1n_2)=1} \psi_M(m) \sum_{\substack{\nu^2+1 \equiv 0 \,\, (m n_0n_1n_2 )}}  e_{mn_0n_1n_2}((h_2n_1-h_1n_2)\nu)\bigg)^{1/2}
\end{align*}
by denoting $n_0=(n_1,n_2)$ and by  using the Chinese Remainder Theorem to collapse the sum over $\nu^2+1 \equiv 0 \,\, (mQ)$ back to a sum over $\nu^2+1\equiv 0 \,\, (mn_0n_1n_2)$.

In the diagonal part $h_1n_2-h_2n_1=0$ we use a trivial estimate to get aboud
\begin{align*}
\pprec M^{1/2} \bigg( \frac{1}{H^2} HNM  \bigg)^{1/2} \ll MN^{1/2}H^{-1/2} \ll x^{1/2}P^{1/2} N^{-1/2} < x^{1-\eta},
\end{align*}
since $H>P/x$ and $N\gg x^{\alpha-1+\eta}$.  

For the off-diagonal $h_1n_2-h_2n_1\neq 0$ we can introduce Kloosterman sums by a similar argument as in  \cite[Section 5]{DI} to get a sum of the type
\begin{align*}
\sum_{r} \sum_{\substack{m \sim \bm{M}\\ n \sim \bm{N}}}A_{m,r} B_{n,r} \sum_{(c,r)=1} g(m,n,c,r) S(m \overline{r},\pm n;c)
\end{align*}
where $g(m,n,c,r)$ is a $C^{\infty}$-smooth function. Here $r$ corresponds to $n_0n_1n_2$, $n$ corresponds to $h_1n_2-h_2n_1$, and $m$ is the frequency parameter that arises from completing an incomplete Kloosterman sum by using Lemma \ref{complete}. Unfortunately both of the coefficients $A_{m,r}$ and $B_{n,r}$ depend on $r$, so that we are unable to make use of the average over the `level variable' $r$ (as in \cite[Theorem 10]{DI2}).  By combining the bound $\theta \leq 7/64$ of Kim and Sarnak \cite[Appendix 2]{KS} with the estimate of Deshouillers and Iwaniec \cite[Theorem 9]{DI2} we can bound
\begin{align*}
 \sum_{\substack{m \sim \bm{M}\\ n \sim \bm{N}}}A_{m,r} B_{n,r} \sum_{(c,r)=1} g(m,n,c,r) S(m \overline{r},\pm n;c)
\end{align*}
for each $r$ individually, which gives a sufficient bound as long as $N \ll x^{(2-2\theta-\alpha)/3}$ for $\theta=7/64$.

\subsection{Sizes of various quantities in the proof}
In the proof of Proposition \ref{typeiiprop}(i) below there will appear numerous quantities. Here we have collected their sizes and relations to one another:
\begin{align*}
&P= x^{\alpha}, \quad \quad MN=P, \quad \quad x^{\alpha-1+\eta} \ll N \ll  x^{(2-2\theta-\alpha)/3-\eta}=x^{(57-32\alpha)/96 - \eta}, \\
&H= x^\epsilon P/x, \quad \quad k \ll M, \quad \quad 1 \ll R, S \ll \frac{P^{1/2}N^{1/2}}{k^{1/2}n_0^{1/2}},  \\
&T = x^\epsilon \frac{S \delta N^2 }{R n_0}, \quad \quad H_1, H_2 \ll H, \quad \quad \varrho = \delta k^2 n_0 n_1n_2 \asymp \delta N^2/n_0, \\
&\bm{M} \ll T, \quad \quad \bm{N} \ll \frac{HN}{kn_0}, \quad \quad \text{and} \quad \quad C \ll S.
\end{align*}

\subsection{Preliminaries}
We have collected here some basic estimates which will be needed in the proof. 
\begin{lemma}  \label{gcdsum} Let $L \geq 1.$ For any integer $q \neq 0$ we have
\begin{align*}
\sum_{1 \leq \ell \leq L} (\ell, q) \leq \tau(q) L.
\end{align*}
\end{lemma}
\begin{proof}
We have
\begin{align*}
\sum_{1 \leq \ell \leq L} (\ell, q) \leq \sum_{d| q} \sum_{1\leq \ell \leq L} 1_{d| \ell} \leq \tau(q)L.
\end{align*}
\end{proof}

The following lemma is easily proved from \cite[Lemma 1]{DI} by using integration by parts multiple times.
\begin{lemma} \emph{\textbf{(Truncated Poisson summation formula).}}  \label{poisson}
Let $\psi$ be a fixed $C^\infty$-smooth compactly supported function and let $x \gg 1$. Let $q \geq 1$ be an integer. Then for any $A, \epsilon > 0$
\begin{align*}
\sum_{n \equiv a \, (q)} \psi\bigg(\frac{n}{x}\bigg) = \frac{1}{q} \int \psi \bigg(\frac{\xi}{x}\bigg) d\xi + \frac{x}{q} \sum_{1 \leq |h| \leq x^\epsilon q/x} \widehat{\psi} \bigg( \frac{h x}{q}\bigg) e \bigg(-\frac{ah}{q} \bigg) + O_{A,\epsilon,\psi}(x^{-A}),
\end{align*}
where $\hat{f}(h):= \int f(\xi)e(h\xi) d\xi$ is the Fourier transform.
\end{lemma}

Applying the above lemma we immediately infer
\begin{lemma} \emph{\textbf{(Completion of sums).}}  \label{complete}
Let $\psi$ be a fixed $C^\infty$-smooth compactly supported function and let $x \gg 1$. Let $q \geq 1$ be an integer. Suppose that $F:\N \to \C$ is a $q$-periodic function. Then for any $A, \epsilon > 0$
\begin{align*}
\sum_{n} \psi\bigg(\frac{n}{x}\bigg) F(n) =  \frac{x}{q} \sum_{0 \leq |h| \leq x^\epsilon q/x} \widehat{\psi} \bigg( \frac{h x}{q}\bigg) \sum_{a \in \Z/q\Z } F(a)e_q (-ah) + O_{A,\epsilon,\psi}\bigg(x^{-A} \sum_{a \in \Z/q\Z} |F(a)| \bigg).
\end{align*}
\end{lemma}

To state the next lemma, for any sequence $a_m$ and any $M > 0$ define the $\ell^2$-norm
\begin{align*}
\|a_M \|_2 := \bigg( \sum_{m \sim M}  |a_m|^2 \bigg)^{1/2}.
\end{align*}
Let $\lambda_1(q)$ denote the smallest eigenvalue of the Laplacian on $\Gamma_0(q)\backslash \mathbb{H}$ (cf. \cite[Section 1]{DI2} for precise definitions). The Selberg eigenvalue conjecture famously states that for every congruence subgroup $\Gamma$ the smallest eigenvalue $\lambda_1(\Gamma)$ is at least 1/4. The current best result towards this is the result of Kim and Sarnak \cite[Proposition 2 in Appendix 2]{KS} which gives the lower bound $\lambda_1(\Gamma)\geq 1/4-(7/64)^2$. By combining this with \cite[Theorem 9]{DI2} of Deshouillers and Iwaniec, we get
\begin{lemma}\emph{\textbf{(Deshouillers-Iwaniec $+$ Kim-Sarnak).}} \label{dilemma} Let $\theta=7/64,$ and let $r$ be a positive integer. Let $C,M,N > 0$ and let $g(m,n,c)$ be a $C^\infty$-smooth function, supported in
\begin{align*}
[M,2M]\times [N,2N] \times [C, 2C]
\end{align*}
and satisfying
\begin{align*}
\bigg| \frac{\partial^{j+k+\ell}}{\partial m^j \partial n^k \partial c^\ell} g(m,n,c)\bigg| \ll M^{-j} N^{-k} C^{-\ell} \quad \text{for} \,\, 0 \leq j,k,\ell \leq 2. 
\end{align*}
 Then for any coefficients $a_m$ and $b_n$ we have
\begin{align*}
\sum_{\substack{m,n,c \\ (c,r)=1}} a_m b_n g(m,n,c) S(m \overline{r},\pm n;c) \pprec&  \bigg(1+ \frac{\sqrt{r}C }{\sqrt{MN}} \bigg)^{2 \theta} \, \mathcal{L} \, \|a_M\|_2 \|b_N\|_2,
\end{align*}
where
\begin{align*}
\mathcal{L} = \frac{(\sqrt{r} C + \sqrt{MN} +\sqrt{M} C )(\sqrt{r} C + \sqrt{MN} +\sqrt{N} C )}{\sqrt{r} C + \sqrt{MN} } .
\end{align*}
\end{lemma}
\begin{remark} In the statement in \cite[Theorem 9]{DI2} there is a typographical error: the factor $(1+ \frac{\sqrt{rC}}{\sqrt{MN}} )$ should be $(1+ \frac{\sqrt{r}C }{\sqrt{MN}})$.
\end{remark}
To apply the above lemma we need an upper bound for the average value of $\|b_N\|_2$:
\begin{lemma} \label{bNaveragelemma} Let $H_1,H_2,N,K \gg 1$ and $H_1\geq H_2$. Then
\begin{align*}
S:= \sum_{k_1,k_2 \sim K} \bigg( \sum_{n \sim N} \bigg| \sum_{\substack{h_1\sim H_1 \\ h_2 \sim H_2}} 1_{h_1k_2-h_2k_1=n} \bigg|^2 \bigg)^{1/2} \, \ll N^{1/2} \max\{ KH_1, K^{3/2} H_1^{1/2} \}.
\end{align*} 
\end{lemma}
\begin{proof}
If $H_1 \geq K$, then trivially $S \ll N^{1/2} K H_1$, since the number of solutions $(h_1,h_2)$ to $h_1k_2-h_2k_1=n$ is bounded by $\ll H_1/k_1 +1 \ll H_1/K$. If $H_1 < K$, then by the Cauchy-Schwarz inequality
\begin{align*}
S &\ll K \bigg( \sum_{\substack{h_1,h_1' \sim H_1 \\ h_2,h_2' \sim H_2}} \sum_{\substack{k_1,k_2 \sim K \\ h_1k_2-h_2 k_1 \sim N}} 1_{k_2(h_1-h_1')=k_1(h_2-h_2')}  \bigg)^{1/2} \\
& \ll K H_1  \bigg( \sum_{\substack{ |\ell_1| \ll H_1 \\ |\ell_2| \ll H_2}}\max_{\substack{h_1 \sim H_1\\ \substack{h_2 \sim H_2}}} \sum_{\substack{k_1,k_2 \sim K \\ h_2 k_1-h_1k_2 \sim N}} 1_{n_2\ell_1 =n_1\ell_2}  \bigg)^{1/2} \\
& \ll K H_1  \bigg( \max_{\substack{h_1 \sim H_1\\ \substack{h_2 \sim H_2}}} \sum_{n\sim N} \sum_{\substack{k_1,k_2 \sim K }} 1_{h_1k_2-h_2k_1=n} \bigg)^{1/2} \ll  K H_1  \bigg( N \frac{K}{H_1}\bigg)^{1/2} = N^{1/2}K^{3/2} H_1^{1/2}.
\end{align*}
\end{proof}

For the proof of Proposition \ref{typeiiprop}(ii) we require the following lemma of de la Bret\`eche and Drappeau \cite[Lemme 8.3, part 1.]{BD} (applied with $r=d=1$ and $D=-1$), which makes explicit the dependence on $\theta$ of the result of Duke, Friedlander and Iwaniec \cite[Proposition 4]{DFI} (for $\theta=1/4$ they give essentially the same result).
\begin{lemma} \label{bdlemma} Let $\theta=7/64$ and fix an integer $q \geq 1$. Suppose that $|h| \leq q$, $M \gg1$, and let $\psi$ be a fixed $C^\infty$-smooth compactly supported function. Then
\begin{align*}
\sum_{(m,q)=1} \psi(m/M) \sum_{\nu^2+1 \equiv 0 \,\, (mq)} e_{mq}(h\nu) \, \pprec |h| \, + \, (q,h)^\theta q^{1/2-\theta}M^{1/2+\theta}.
\end{align*}
\end{lemma}

\subsection{Evaluation of $|\A_{mn}|$ by Poisson Summation} \label{poissonsection}
We are now in place to begin the proof of Proposition \ref{typeiiprop}. We will first show part (i) and in the end part (ii). By the Truncated Poisson summation formula (Lemma \ref{poisson}) we have for any $\epsilon > 0$
\begin{align*}
|\A_{mn}| &= \sum_{\ell^2+1 \equiv 0 \,\, (mn)} b(\ell) = \sum_{\substack{\nu \,\, (mn) \\ \nu^2+1 \equiv 0 \,\, (mn)}} \sum_{\ell \equiv \nu \,\, (mn)} b(\ell) \\
& = \frac{\rho(mn)}{mn} X + r(\A,mn) + O_{A,\epsilon} (x^{-A} ),
\end{align*}
where, for $\psi(z):=b(xz)$ and $H:= x^\epsilon P/x$, we have
\begin{align*}
r(\A,mn) = \frac{x}{mn} \sum_{1 \leq |h| \leq H} \widehat{\psi}(hx/mn)\sum_{\substack{\nu \,\, (mn) \\ \nu^2+1 \equiv 0 \,\, (mn)}} e_{mn}(-h\nu).
\end{align*}
The smooth `cross-conditions' $\widehat{\psi}(hx/mn)$ and $\psi_P(mn)\log mn$ may be removed by applying Mellin transform (similarly as one can use Perron's formula to remove cross-conditions as in \cite[Chapter 3.2]{harman}). Hence, Proposition \ref{typeiiprop} follows once we show
\begin{prop} \label{expprop} Let $c_h$ be any bounded coefficients. Adopting the assumptions of Proposition \ref{typeiiprop}, for $H:=x^\epsilon P/x$ we have
\begin{align} \label{expsum}
\Sigma(M,N):=\frac{1}{H} \sum_{1 \leq |h| \leq H} c_h \sum_{\substack{m \sim M \\ n \sim N}} a_m b_n \sum_{\substack{\nu \,\, (mn) \\ \nu^2+1 \equiv 0 \,\, (mn)}} e_{mn}(-h\nu) \, \ll x^{1-\eta}.
\end{align}
\end{prop}
Our proof of Proposition \ref{typeiiprop}(i) actually gives the following general bound, which we state only in the case $H \ll N $ for simplicity.
\begin{prop}  Let $M,N,H \geq  1$ with $H\ll N $ and let $a_m$, $b_n$ and $c_h$ be divisor-bounded coefficients. Assume that $b_n$ is supported on square-free integers. Then
\begin{align*}
\frac{1}{H} \sum_{1 \leq |h| \leq H} c_h & \sum_{\substack{m \sim M \\ n \sim N}} a_m b_n \sum_{\substack{\nu \,\, (mn) \\ \nu^2+1 \equiv 0 \,\, (mn)}} e_{mn}(-h\nu) \\
 &\pprec \frac{M N^{1/2}}{H^{1/2}} + \sqrt{HMN} + H^{1/2}M^{1/4}N + M^{3/4}N^{1/2} + \frac{M^{3/4+\theta/2}N^{3/2+\theta/2}}{H^{1/2+\theta/2}} . 
\end{align*}
\end{prop}
\subsection{Application of the Cauchy-Schwarz inequality} \label{cssection}
Let us write $k=(m,n)$ and make the change of variables $m \mapsto km$ and $n \mapsto kn$ to get 
\begin{align*}
\Sigma(M,N)= \sum_{k\ll N} \Sigma_{k}(M,N)
\end{align*}
for
\begin{align*}
\Sigma_k(M,N):= \sum_{\substack{m \sim M/k}} a_{km}  \frac{1}{H} \sum_{1 \leq |h| \leq H} c_h \sum_{\substack{n \sim N/k\\ (n,km)=1}}b_{kn}  \sum_{\substack{\nu^2+1 \equiv 0 \,\, (k^2mn)}} e_{k^2mn}(-h\nu).
\end{align*}
We will show that $\Sigma_k(M,N) \pprec x^{1-\eta}/k$ (in the first pass the reader may wish to restrict to the case $k=1$). Before applying the Cauchy-Schwarz inequality we note that by the Chinese Remainder Theorem for any coprime integers  $a,b$ the solutions to $\nu^2+1 \equiv 0 \,(ab)$ are in one-to-one correspondence to the solutions to the pair of equations $\alpha^2+1 \equiv 0 \,(a), \,\beta^2+1 \equiv 0 \,(b)$. Thus, denoting 
\begin{align*}
Q=Q(km):= \prod_{\substack{2 \leq p\leq 2N \\ p \equiv 1,2 \, \, (4) \\ p\, \nmid\, km}} p,
 \end{align*}
we have
\begin{align*}
  \sum_{\substack{\nu^2+1 \equiv 0 \,\, (k^2mn)}} e_{k^2mn}(-h\nu) =\frac{\rho(n)}{\rho(Q)} \sum_{\substack{\nu^2+1 \equiv 0 \,\, (k^2 m Q)}} e_{k^2mn}(-h\nu)
\end{align*}
by using the fact that $b_n$ is supported on square-free integers.  Inserting this and applying the Cauchy-Schwarz inequality we get 
\begin{align} \nonumber
\Sigma_k(M,N) \pprec & \frac{\sqrt{M}}{\sqrt{k}} \bigg(   \sum_m \psi_{M}(km)   \frac{1}{\rho(Q)} \sum_{\substack{\nu^2+1 \equiv 0 \,\, (k^2mQ)}}   \bigg| \frac{1}{H} \sum_{h} c_h \sum_{\substack{n \sim N\\(n,m)=1}} b_{kn} \rho(n) e_{k^2mn}(-h\nu) \bigg|^2 \bigg)^{1/2} \\ \label{cs}
=& \frac{\sqrt{M}}{\sqrt{k}}  \bigg( \frac{1}{H^2} \sum_{1\leq |h_1|,|h_2| \leq H} c_{h_1} \overline{c_{h_2}}   \sum_m \psi_{M}(km)  \sum_{\substack{n_1,n_2 \sim N/k\\ (n_1n_2,m)=1}} b_{kn_1}  \overline{b_{kn_2}} \\ \nonumber
& \hspace{100pt}  \frac{\rho(n_1)\rho(n_2)}{\rho(Q)} \sum_{\substack{\nu^2+1 \equiv 0 \,\, (k^2mQ)}}  e_{k^2mn_1}(-h_1\nu) e_{k^2mn_2}(h_2\nu) \bigg)^{1/2}.
\end{align}
Denote $n_0:= (n_1,n_2),$ and make the change of variables $n_j \mapsto n_0n_j$ in the above sum.  Since $n_0n_1n_2$ is square-free and coprime to $km$, by the Chinese Remainder Theorem we obtain
\begin{align*}
\frac{\rho(n_0 n_1)\rho(n_0 n_2)}{\rho(Q)} \sum_{\substack{\nu^2+1 \equiv 0 \,\, (k^2 m Q)}}  & e_{k^2mn_0n_2}(h_2\nu)e_{k^2mn_0 n_1}(-h_1\nu) \\
&= \frac{\rho(n_0 n_1)\rho(n_0 n_2)}{\rho(Q )} \sum_{\substack{\nu^2+1 \equiv 0 \,\, (k^2m Q )}}  e_{k^2mn_0n_1n_2}((h_2n_1-h_1n_2)\nu) \\
& = \rho(n_0) \sum_{\substack{\nu^2+1 \equiv 0 \,\, (k^2 m n_0n_1n_2 )}}  e_{k^2mn_0n_1n_2}((h_2n_1-h_1n_2)\nu).
\end{align*}
Hence, we obtain $\Sigma(M,N)_k^2 \pprec (M/k) \cdot \Xi_k(M,N)$, where
\begin{align*}
\Xi_k(M,N):=  \frac{1}{H^2}& \sum_{1\leq |h_1|,|h_2| \leq H} c_{h_1} \overline{c_{h_2}}   \sum_{n_0\ll N} \rho(n_0) \sum_{\substack{n_1,n_2 \sim N/kn_0 \\ (n_1,n_2)=1}} b_{kn_0n_1}  \overline{b_{kn_0 n_2}} \\
& \sum_{(m,n_0n_1n_2)=1} \psi_{M}(km) \sum_{\substack{\nu^2+1 \equiv 0 \,\, (k^2 m n_0n_1n_2 )}}  e_{k^2mn_0n_1n_2}((h_2n_1-h_1n_2)\nu) .
\end{align*}

We immediately note that the contribution from the diagonal $h_1n_2-h_1n_2 =0$ to $\Xi_k(M,N)$ is trivially bounded by
\begin{align*}
\pprec \frac{M}{kH^2} \sum_{n_0 \ll N} \sum_{1 \leq |h_1|,|h_2| \leq 2 H} \sum_{n_1,n_2 \ll N/kn_0} 1_{h_1n_2=h_2n_1} \pprec \frac{MN}{k H},
\end{align*}
which contributes to $\Sigma_k(M,N)$  at most
\begin{align} \label{diagtrue}
\pprec  \frac{1}{k} M^{1/2} \bigg(\frac{MN}{H} \bigg)^{1/2} = \frac{MN^{1/2}}{kH^{1/2}} \ll  \frac{1}{k} x^{1/2} P^{1/2} N^{-1/2} \ll x^{1-\eta}/k
\end{align}
by using $H=x^\epsilon P/x$ and the assumption $N \gg x^{\alpha-1+\eta}$. Therefore, we may assume below that $h_1n_2-h_2n_1 \neq 0$.

\subsection{Introducing Kloosterman sums}
We expand the condition $(m, n_0 n_1n_2)=1$ by using the M\"obius function to get
\begin{align*}
 \sum_{(m, n_0n_1n_2) = 1} = \sum_{\delta |  n_0 n_1n_2} \mu (\delta) \sum_{\substack{m \\ \delta | m}}.
\end{align*}
In the first pass the reader may wish to pretend that $\delta=1$ below. Let us denote $\ell := m k^2 n_0 n_1n_2,$ so that the condition $\delta | m$ can be written as $\delta k^2 n_0 n_1 n_2 | \ell$ and 
\begin{align*}
\Xi_k(M,N)=& \frac{1}{H^2} \sum_{1\leq |h_1|,|h_2| \leq H} c_{h_1} \overline{c_{h_2}} \sum_{n_0\ll N} \rho(n_0) \sum_{\substack{n_1,n_2 \sim N/kn_0 \\ (n_1,n_2)=1 \\ h_1n_2-h_2n_1 \neq 0}} b_{kn_0n_1} \overline{b_{kn_0n_2}} \sum_{\delta| n_0n_1n_2}\mu(\delta)  \\
 & \sum_{\ell  \equiv 0 \, \, (\delta k^2 n_0 n_1n_2)} \psi_{M}\bigg(\frac{\ell}{k n_0n_1n_2}\bigg) \sum_{\substack{\nu^2+1\equiv 0 \, (\ell)}}  e_{\ell} ((h_2n_1 -h_1 n_2)\nu ) + O(MN/kH).
\end{align*}
 The variable $\ell$ is of size $ P N / k n_0.$ To proceed we require the following Lemma of Gauss (cf. \cite[Lemma 2]{DI}):
\begin{lemma} \label{gauss} If the equation $\nu^2+1 \equiv 0 \,\, (\ell)$ has a solution, then $\ell$ has a representation as a sum of two squares
\begin{align*}
\ell=r^2+s^2, \quad (r,s)=1, \quad r,s >0.
\end{align*}
Furthermore, there is a one-to-one correspondence between such representations and the solutions to $\nu^2+1 \equiv 0 \,\, (\ell),$ and we have
\begin{align*}
\frac{\nu }{\ell} \equiv \frac{r}{s(r^2+s^2)} - \frac{\overline{r}}{s}  \mod 1.
\end{align*}
\end{lemma}
Applying this lemma we get
\begin{align*}
e_{\ell} \bigg((h_2n_1 -h_1 n_2)\nu \bigg) = e_{s} \bigg(\frac{h_1n_2 -h_2 n_1}{r} \bigg)\bigg(1+ O \bigg( \frac{Hr}{Ps}\bigg) \bigg).
\end{align*}
The contribution from the $O$-term to $\Sigma_k(M,N)$ is trivially bounded by
\begin{align} \label{nonarithm}
&  \frac{\sqrt{M}}{\sqrt{k}} \bigg(  \frac{H}{P} \sum_{n_0 \ll N} \sum_{\substack{n_1,n_2 \sim N/kn_0}} \sum_{\delta| n_0n_1n_2}  \max_t \sum_{s \ll(PN/kn_0)^{1/2}} \frac{1}{s} \sum_{\substack{r \ll (PN/kn_0)^{1/2} \\ r \equiv t \, (\delta k^2 n_0 n_1n_2)}} r \bigg)^{1/2} \\ \nonumber
& \pprec \frac{\sqrt{M}}{\sqrt{k}} \bigg(  \frac{H}{P}  \sum_{n_0 \ll N} \sum_{\substack{n_1,n_2 \sim N/kn_0}} \sum_{\delta| n_0 n_1n_2} \frac{P^{1/2} N^{1/2}}{(n_0k)^{1/2}}\bigg( \frac{P^{1/2}N^{1/2}}{k^{5/2}\delta n_0^{3/2} n_1n_2 } + 1\bigg) \bigg)^{1/2} \\ \nonumber
&\pprec \frac{1}{k} ( M^{1/2} H^{1/2} N^{1/2} + M^{1/2} H^{1/2} N^{5/4}P^{-1/4}) \\ \label{nonarithmcontribution}
& =  \frac{x^{\epsilon/2}}{k}  (Px^{-1/2}  + P^{3/4}N^{3/4}x^{-1/2})\ll x^{1-\eta}/k
\end{align}
since from the assumptions it follows that $\alpha < 3/2-\eta$ and $N < x^{2-\alpha - \eta}.$

Hence, we have $\Xi_k(M,N)= \widetilde{\Xi_k}(M,N) +O(\mathcal{E}),$ where $(M/k)^{1/2} \mathcal{E}^{1/2} < x^{1-\eta}/k$ and
\begin{align*}
\widetilde{\Xi_k}(M,N):= \frac{1}{H^2} \sum_{1\leq |h_1|,|h_2| \leq H} c_{h_1} \overline{c_{h_2}} & \sum_{n_0\ll N} \rho(n_0) \sum_{\substack{n_1,n_2 \sim N/kn_0 \\ (n_1,n_2)=1 \\ h_1n_2-h_2n_1 \neq 0}} b_{kn_0n_1} \overline{b_{kn_0n_2}} \sum_{\delta| n_0n_1n_2}\mu (\delta) \\
 &\sum_{\substack{r,s > 0\\ (r,s)=1\\r^2 \equiv -s^2 \, (\delta k^2 n_0n_1n_2)}} \psi_{M}\bigg(\frac{r^2+s^2}{k n_0n_1n_2}\bigg)  e_{s} \bigg(\frac{h_1n_2 -h_2 n_1}{r} \bigg).
\end{align*}

\subsection{Completing the sum}
By a smooth dyadic partition of unity for the variables $r$ and $s$, we can split $\widetilde{\Xi_k}(M,N)$ into $\ll \log^2x$ sums of the form 
\begin{align*}
\Psi_k(R,S):=  \frac{1}{H^2} \sum_{1\leq |h_1|,|h_2| \leq H} c_{h_1} \overline{c_{h_2}} & \sum_{n_0\ll N} \rho(n_0) \sum_{\substack{n_1,n_2 \sim N/kn_0 \\ (n_1,n_2)=1 \\ h_1n_2-h_2n_1 \neq 0}} b_{kn_0n_1} \overline{b_{kn_0n_2}} \sum_{\delta| n_0n_1n_2}\mu (\delta) \\
 &\sum_{\substack{(r,s)=1\\r^2 \equiv -s^2 \, (\delta k^2 n_0n_1n_2)}} g(r,s,n_0n_1n_2)  e_{s} \bigg(\frac{h_1n_2 -h_2 n_1}{r} \bigg).
\end{align*}
where 
\begin{align*}
g(r,s,n_0n_1n_2) := \psi_R(r) \psi_S(s) \psi_{M}\bigg(\frac{r^2+s^2}{k n_0n_1n_2}\bigg) 
\end{align*}
with $\psi_R(r)$ (similarly for $\psi_S(s)$) a $C^\infty$-smooth function supported on $[R,2R]$ and satisfying $\psi_R^{(i)}(r) \ll_i R^{-i}$ for all $i\geq 0$, where 
\begin{align*}
1 \ll R, S \ll \frac{P^{1/2}N^{1/2}}{k^{1/2} n_0^{1/2}} \quad \text{and} \quad \max\{R,S\} \gg \frac{P^{1/2}N^{1/2}}{k^{1/2} n_0^{1/2}}.
\end{align*}
 For each $R$ and $S$ we can now complete the sum over $r$ by using the Poisson summation formula (Lemma \ref{complete}), similarly as in \cite[Section 5]{DI}. The modulus of the sum is of size $S\delta k^2 n_0n_1n_2 \asymp S \delta  N^2/n_0$, and the length of the sum is $R,$ so that for 
\begin{align*}
T:= x^{\epsilon} \frac{S  \delta N^2}{R n_0 }
\end{align*}
 we get  by Lemma \ref{complete}
\begin{align} \nonumber
& \sum_{\substack{r \\(r,s)=1\\ r^2 \equiv -s^2 \, (\delta k^2 n_0 n_1n_2)}}  g(r,s,n_0 n_1n_2)  e_{s} \bigg(\frac{h_1n_2 -h_2 n_1}{r} \bigg) +O_{A,\epsilon}(x^{-A}) \\ \label{precomplete}
& \hspace{25pt}= \frac{x^\epsilon}{T} \sum_{|t| \leq T} G(t,s,n_0 n_1n_2)\sum_{\substack{u \, (s\delta k^2 n_0n_1n_2)\\ (u,s)=1 \\ u^2\equiv -s^2 \, (\delta k^2 n_0 n_1n_2)}} e_{s} \bigg(\frac{h_1n_2 -h_2 n_1}{u} \bigg) e_{s\delta k^2 n_0 n_1n_2}(-tu), 
\end{align}
where
\begin{align*}
G(t,s,n_0 n_1n_2) = \frac{R T}{ x^\epsilon s\delta k^2 n_0n_1n_2} \widehat{f }_{s,n_0,n_1,n_2}(tR/ s\delta k^2 n_0n_1n_2)
\end{align*}
for
\begin{align*}
f_{s,n_0,n_1,n_2}(x) := g(Rx,s,n_0n_1n_2)
\end{align*}
(so that the function $G$ is bounded). By writing $u=\alpha s + \beta \delta k^2 n_0n_1n_2$ (note that $(u,s)=1$ implies $(s,\delta k^2 n_0n_1n_2)=1$) the right-hand side in (\ref{precomplete}) is equal to
\begin{align*}
\frac{x^\epsilon}{T} \sum_{|t| \leq T} G(t,s,n_0 n_1n_2) \sum_{\alpha^2+1 \equiv 0 \, (\delta k^2 n_0 n_1n_2)} e_{\delta k^2 n_0 n_1n_2}(-t \alpha) \sum_{\substack{\beta \, (s) \\ (\beta,s)=1}} e_{s} \bigg(\frac{h_1n_2 -h_2 n_1}{\delta k^2 n_0n_1n_2\beta} - t \beta \bigg). 
\end{align*}

The contribution from $t=0$ to $\Psi_k(R,S)$ is by a standard bound for Ramanujan's sums bounded by (using Lemma \ref{gcdsum})
\begin{align*}
 \pprec \frac{1}{T H^2}\sum_{h_1,h_2}  \sum_{n_0 \ll N}  \sum_{\substack{n_1,n_2 \sim N/kn_0 \\ (n_1,n_2)=1\\ h_1n_2-h_2n_1 \neq 0}}  \sum_{s} (h_1n_2-h_2 n_1, s)  \pprec \sum_{n_0 \ll N}  \frac{S N^2}{Tk^2n_0^2}   \ll P^{1/2} N^{1/2} k^{-2}.
\end{align*}
The contribution from this to $\Sigma_k(M,N)$ is
\begin{align} \label{t0contribution}
\pprec M^{1/2} P^{1/4} N^{1/4}/k = P^{3/4} N^{-1/4}/k \ll x^{1-\eta}/k
\end{align}
since $N \gg x^{\alpha -1+ \eta} \gg x^{3\alpha-4 + \eta}$ for $\alpha < 3/2$.

 Therefore, the sum $\Psi_k(R,S)$ is up to a negligible error term equal to a sum of Kloosterman sums of the form
\begin{align*}
&\widetilde{\Psi_k}(R,S):= \frac{x^\epsilon}{T H^2}\sum_{1 \leq |h_1|,|h_2|\leq H} c_{h_1} \overline{c_{h_2}}  \sum_{n_0\ll N} \rho(n_0) \sum_{\substack{n_1,n_2 \sim N/kn_0 \\ (n_1,n_2)=1 \\ h_1n_2-h_2n_1 \neq 0}} b_{n_0n_1} \overline{b_{n_0n_2}} \sum_{\delta| n_0n_1n_2} \mu (\delta) \\
 &\sum_{\alpha^2+1 \equiv 0 \, (\delta k^2 n_0 n_1n_2)} \sum_{1\leq |t| \leq T} e_{\delta k^2 n_0 n_1n_2}(-t \alpha) \sum_{(s,\delta k^2 n_0n_1n_2)=1}  G(t,s,n_0 n_1n_2)S(- t\overline{\delta k^2 n_0 n_1n_2}, h_1n_2-h_2n_1; s).
\end{align*}

\subsection{Application of the Deshouillers-Iwaniec bound} \label{disection}
We split the sum over $h_1$ and $h_2$ dyadically to parts with $h_1 \sim H_2$ and $ h_2 \sim H_2$. By symmetry we may assume $H_1 \geq H_2.$ We now fix $n_0,n_1,n_2, \delta,\alpha$, and write
\begin{align*}
\varrho:= \delta k^2 n_0 n_1n_2 \asymp \delta N^2 /n_0
\end{align*}
and (denoting $m:=t$ and $n:= h_1n_2-h_1n_2$)
\begin{align*}
A_m = A_m(\varrho,\alpha):= e_{\varrho}(-m\alpha) \quad \text{and} \quad B_n= B_n(n_1,n_2) :=   \sum_{\substack{h_1 \sim H_1 \\ h_2 \sim H_2 \\ n= h_1 n_2-h_2n_1}}c_{h_1} \overline{c_{h_2}}.
\end{align*}
\begin{remark} Since both of the coefficients $A_m$ and $B_n$ depend on the level $r$, we are unable to make use of the average over $r$ as in \cite[Theorem 10]{DI2}.
\end{remark}
Since $t \neq 0 \neq h_1n_2-n_2h_1$, by a smooth dyadic decomposition in  the variables $m$ and $n$ we can partition $\widetilde{\Psi_k}(R,S)$ into $\ll \log^2 x$ sums of the form
\begin{align*}
\Upsilon_k := \frac{1}{TH^2}   \sum_{n_0\ll N} \rho(n_0) \sum_{n_1,n_2 \sim N/kn_0} \sum_{ \substack{\delta | n_0n_1n_2}} \max_{\alpha \, (\varrho)} \bigg| \sum_{\substack{m,n,c \\ (c,\varrho)=1}} A_m B_n F(m,n,c) S(m \overline{\varrho}, \pm n;c) \bigg|,
\end{align*}
where it is easily verified that $F$ satisfies the assumptions of Lemma \ref{dilemma} in the range
\begin{align*}
(m,n,c) \in [\bm{M},2\bm{M}] \times [\bm{N},2\bm{N}] \times [C,2C]
\end{align*}
for
\begin{align*}
 \bm{M} \ll T = x^{\epsilon} \frac{S  \delta N^2}{R n_0 }, \quad  \bm{N} \ll HN/kn_0, \quad \text{and} \quad  C \ll  S \ll \frac{P^{1/2}N^{1/2}}{k^{1/2} n_0^{1/2}}.
\end{align*}
By Lemma \ref{dilemma} we get  for $\theta:=7/64$
\begin{align*}
 \bigg| \sum_{\substack{m,n,c\\(c,\varrho)=1}} A_m B_n F(m,n,c) S(m \overline{\varrho}, \pm n;c) \bigg| \pprec \bigg(1+ \frac{\sqrt{\varrho}C}{\sqrt{\bm{M}\bm{N}}} \bigg)^{2\theta} \mathcal{L} \, \| A_{\bm{M}}\|_2 \| B_{\bm{N}}\|_2 
\end{align*}
for 
\begin{align*}
\mathcal{L} &= \frac{(\sqrt{\varrho}C +\sqrt{\bm{M}\bm{N}} + \sqrt{\bm{M}} C) (\sqrt{\varrho}C +\sqrt{\bm{M}\bm{N}} + \sqrt{\bm{N}} C) }{\sqrt{\varrho}C +\sqrt{\bm{M}\bm{N}} } \\
& = \sqrt{\varrho}C +\sqrt{\bm{M}\bm{N}} + \sqrt{\bm{M}} C+ \sqrt{\bm{N}} C+ \frac{\sqrt{\bm{M} \bm{N}} C^2}{\sqrt{\varrho}C +\sqrt{\bm{M}\bm{N}}} \\
& \ll \sqrt{\varrho} C + \sqrt{\bm{M}} C + \sqrt{\bm{M} \bm{N}} \ll \sqrt{\delta/n_0 } N S + \sqrt{T} S + \sqrt{T HN/kn_0},
\end{align*}
where the last bound follows from $  \bm{N} \ll \varrho $. We have $\|A_{\bm{M}}\|_2 \, \ll \sqrt{\bm{M}}$, and by Lemma \ref{bNaveragelemma} 
\begin{align*}
\sum_{n_1,n_2 \sim N/kn_0} \| B_{\bm{N}}\|_2 \, \ll \sqrt{\bm{N}} \max \bigg\{ \frac{N H_1}{k n_0}, \frac{N^{3/2} H_1^{1/2}}{k^{3/2}n_0^{3/2}}\bigg\}.
\end{align*}
Hence, by using $H \ll N$ we have
\begin{align*}
 \Upsilon_k &\pprec \, \max_{\delta } \sum_{n_0 \ll N} \frac{1}{TH^2}\bigg(1+ \frac{\sqrt{\varrho }C}{\sqrt{\bm{M}\bm{N}}} \bigg)^{2\theta}(\sqrt{\delta/n_0 } N S + \sqrt{T} S + \sqrt{T HN/kn_0}) \\
 & \hspace{200pt} \cdot \sqrt{\bm{M}\bm{N}}\max \bigg\{ \frac{N H_1}{kn_0}, \frac{N^{3/2} H_1^{1/2}}{k^{3/2}n_0^{3/2}}\bigg\} \\
 &\pprec \, \max_{\delta } \sum_{n_0 \ll N} \frac{1}{TH^2}\bigg(1+ \frac{\sqrt{k \delta }N S}{\sqrt{T HN}} \bigg)^{2\theta}(\sqrt{\delta/n_0 } N S + \sqrt{T} S + \sqrt{T HN/kn_0}) \sqrt{T} \frac{H N^2}{k^{3/2}n_0^{3/2}} \\
  &\ll \max_{\delta }\sum_{n_0 \ll N} \frac{1}{kn_0^{3/2}} \bigg(1+ \frac{\sqrt{\delta }N S}{\sqrt{T HN}} \bigg)^{2\theta} \bigg( \frac{\sqrt{\delta/n_0 } SN^3}{ \sqrt{T} H} + \frac{SN^2}{ H } + \frac{N^{5/2}}{\sqrt{Hn_0}} \bigg),
\end{align*}
since the first bound is increasing as a function of $\bm{M}$ and $\bm{N}.$ Inserting $T =x^\epsilon S  \delta N^2 / R n_0 $ we obtain
\begin{align*}
 \Upsilon_k & \pprec \sum_{n_0 \ll N} \frac{1}{kn_0^{3/2}} \bigg(1+ \frac{ \sqrt{n_0  RS}}{\sqrt{H N }} \bigg)^{2\theta} \bigg( \frac{ \sqrt{ RS} N^2}{ H} + \frac{SN^2}{ H } + \frac{N^{5/2}}{\sqrt{H}} \bigg) \\
 &\pprec \frac{1}{k}\bigg(1+ \frac{\sqrt{P}}{\sqrt{H}} \bigg)^{2\theta}\bigg(  \frac{P^{1/2} N^{5/2}}{ H } + \frac{N^{5/2}}{\sqrt{H}} \bigg),
\end{align*}
since $R,S \ll P^{1/2}N^{1/2}n_0^{-1/2}.$ By using $H= x^\epsilon P/x$ this yields
\begin{align} \label{upsilonbound}
\Upsilon_k \pprec  \frac{x^{1+\theta} N^{5/2}}{kP^{1/2}},
\end{align}
so that the contribution to $\Sigma_k (M,N)$ is 
\begin{align} \label{dicontribute}
\pprec M^{1/2}  x^{1/2+\theta/2} N^{5/4}P^{-1/4}/k = x^{1/2+\theta/2}P^{1/4}N^{3/4}/k \ll x^{1-\eta}/k 
\end{align}
by using the assumption $N \ll x^{(2-2\theta-\alpha)/3-\eta}$.

\subsection{Proof of Proposition \ref{typeiiprop}(i)}
By combining the bounds (\ref{diagtrue}), (\ref{nonarithmcontribution}), (\ref{t0contribution}), and (\ref{dicontribute}) we obtain $\Sigma_k(M,N) \ll x^{1-\eta}/k$. Summing over $k \ll M$ we get $\Sigma(M,N) \ll x^{1-\eta}$, which by Section \ref{poissonsection} proves Proposition \ref{typeiiprop}(i). \qed

\subsection{Proof of Proposition \ref{typeiiprop}(ii)}
We need to prove (\ref{expsum}) under the assumptions in Proposition \ref{typeiiprop}(ii). We use a similar argument as in \cite[Section 5]{DFI} with Lemma \ref{bdlemma}. Inserting the condition $(m,n)=1$ to $\Sigma(M,N)$ gives an error term (since $b_n$ are supported on primes)
\begin{align*}
 \pprec \sum_{n \sim N} \sum_{m \sim M} 1_{n|m} \pprec M \ll x^{1-\eta},
\end{align*}
so that we may restrict to the part $(m,n)=1$. Applying the Cauchy-Schwarz inequality similarly as in Section \ref{cssection} but with the sum over $h$ `outside' we get $\Sigma(M,N) \pprec M^{1/2}\cdot \Xi(M,N)^{1/2} $ for
\begin{align*}
\Xi(M,N):=  \frac{1}{H} \sum_{1\leq |h| \leq H} &  \sum_{n_0\ll N} \rho(n_0) \sum_{\substack{n_1,n_2 \sim N/n_0 \\ (n_1,n_2)=1}} b_{n_0n_1}  \overline{b_{n_0 n_2}} \\
& \sum_{(m,n_0n_1n_2)=1} \psi_M(m) \sum_{\substack{\nu^2+1 \equiv 0 \,\, (m n_0n_1n_2 )}}  e_{mn_0n_1n_2}(h(n_1-n_2)\nu).
\end{align*}
The diagonal part $n_1=n_2$ is bounded by $\pprec M N$, whose contribution to $\Sigma(M,N)$ is at most $\pprec MN^{1/2} < x^{1-\eta}$ by using $N \gg x^{2(\alpha-1)+\eta}$. For the remaining part $\Xi_0(M,N)$ with $n_0=1$ we use Lemma \ref{bdlemma} with $q=n_1n_2$ to get
\begin{align*}
\Xi_0(M,N) := & \frac{1}{H} \sum_{1\leq |h| \leq H} \sum_{\substack{n_1,n_2 \sim N \\ (n_1,n_2)=1}} b_{n_1}  \overline{b_{n_2}} \sum_{(m,n_1n_2)=1} \psi_M(m) \sum_{\substack{\nu^2+1 \equiv 0 \,\, (m n_1n_2 )}}  e_{mn_1n_2}(h(n_1-n_2)\nu) \\
\ll & \sum_{\substack{n_1,n_2 \sim N \\ (n_1,n_2)=1}}  \frac{1}{H} \sum_{1\leq |h| \leq H} \bigg(HN + (n_1n_2,h(n_1-n_2))^\theta N^{1-2\theta}M^{1/2+\theta}\bigg) \\
\pprec & \, HN^3 + M^{1/2+\theta} N^{3-2\theta}
\end{align*}
by computing the sum over $h$ with Lemma \ref{gcdsum}. The contribution from this to $\Sigma(M,N)$ is bounded by
\begin{align*}
\pprec M^{1/2}H^{1/2} N^{3/2} + M^{3/4+\theta/2} N^{3/2-\theta} = x^{\epsilon/2} P N x^{-1/2} +P^{3/4+\theta/2} N^{3/4-3\theta/2} \ll x^{1-\eta},
\end{align*}
since $N \ll x^{(4-(3+2\theta)\alpha)/(3-6\theta)-\eta} < x^{3/2-\alpha-\eta}$. 
Hence, $\Sigma(M,N) \ll x^{1-\eta}$. \qed

\section{Remarks on the arithmetic information}
For $\alpha=1+o(1)$ Proposition \ref{typeiiprop}(i) gives Type II information for $N \ll x^{1/3-2\theta/3-\eta}$, while part (ii) works for $N \ll x^{1/3-\eta}$. The reason for this discrepancy is that we were unable to use the average over the level variable $r$ in Section \ref{disection}. If we could use the average over $r$, we expect that the dependency on the parameter $\theta$ would be same as in \cite[Lemme 8.3, part 3.]{BD}, that is, $M^{\theta}Q^{-\theta},$ where $Q$ corresponds to $N^2$ (note that by a more careful argument we know that the coefficient $c_h$ is a nice smooth function of $h$). Therefore, instead of (\ref{upsilonbound}), our bound for $\Upsilon_k$ would read $xM^\theta N^{5/2-2\theta}P^{-1/2}/k$, which yields
\begin{conj} Suppose that $\alpha < 3/2-\eta$. Let $H=x^\epsilon P/x$ and let $c_h=\psi(h/H)$  for some fixed compactly supported $C^\infty$-smooth function $\psi$. Then for $b_n$ supported on square-free integers we have
\begin{align*}
\Sigma(M,N) \pprec x^{1/2} M^{1/2} + x^{1/2}M^{1/4+\theta/2}N^{1-\theta}+x^{1-\eta}.
\end{align*}
\end{conj}

This gives a bound $\Sigma(M,N) \ll x^{1-\eta}$ as soon as 
\begin{align*}
x^{\alpha-1+\eta}\ll N\ll x^{(2-(1+2\theta)\alpha)/(3-6\theta) - \eta}.
\end{align*}
Note that this is better than the combined bound of Proposition \ref{typeiiprop} parts (i) and (ii), and for $\alpha=1+o(1)$ the upper limit is $x^{1/3-\eta}$. Assuming the above bound with $\theta=7/64$ we can improve the exponent in Theorem \ref{maint} from $1.279$ to $1.286$.

The main reason why the Type II estimate is restricted to small values of $P$ is that we have to use the Cauchy-Schwarz inequality, which means that all savings are essentially halved. Therefore, for large $P$ one should attempt to obtain some other type of arithmetical information where the Cauchy-Schwarz inequality is not necessary, eg. an asymptotic for Type I$_2$ sums
\begin{align*}
\sum_{d \leq D_2} \lambda_d \sum_{\substack{m\sim M, \, \, n \sim N \\
mn \equiv 0 \, (d)}}  |\A_{mn}| \psi_P(mn) \log mn 
\end{align*}
where the most important range would be $M=N=\sqrt{P}$. Even for $D_2=1$ this is an open problem. 

Currently we have an asymptotic formula for $S(x,P)$ only in the range $P=x^{1+o(1)}$ (this follows already from the work of Duke, Friedlander, and Iwaniec \cite{DFI}). To get an asymptotic formula for $S(x,P)$ with $P$ up to $x^{1+\beta}$ for some fixed $\beta>0$ it seems that we would need to handle also Type I$_3$ sums of the form
\begin{align*}
\sum_{d \leq D_3} \lambda_d \sum_{\substack{\ell \sim L, \,\, m\sim M, \, \, n \sim N \\
\ell mn \equiv 0 \, (d)}}  |\A_{\ell mn}| \psi_P(\ell mn) \log \ell mn.
\end{align*}
This is because in Section \ref{alphasmallsection} the sums that we cannot handle are
\begin{align*}
\sum_{x^\sigma < q \leq x^{\alpha-2\sigma}} S(\A(P)_q, q) \quad \quad \text{and} \quad \quad \sum_{U < q \leq x^{\alpha/2}} S(\A(P)_q, q),
\end{align*}
where the first sum corresponds to a sum of three primes all of size $x^{\alpha/3+O(\beta)}$, and the second sum is a sum over two primes of size $x^{\alpha/2+O(\beta)}$. 

\bibliography{quadraticbib}
\bibliographystyle{abbrv}

\end{document}